\documentclass[12pt,reqno]{amsart}

\usepackage{fixltx2e} 
\usepackage[utf8]{inputenc} 
\usepackage{amssymb, latexsym, stmaryrd, amsthm, dsfont, amsfonts, amsbsy, amsmath, mathrsfs} 
\usepackage{mathtools} 
\usepackage{bm, bbm} 
\usepackage{enumerate} 
\usepackage{verbatim} 
\usepackage{url}   
\usepackage{lscape} 
\usepackage{centernot}
\usepackage[dvipsnames]{xcolor}
\usepackage[colorinlistoftodos]{todonotes} 
\usepackage{nicefrac} 
\usetikzlibrary{calc}
\usepackage{todonotes}

\usepackage{microtype,tikz,tikz-cd}  
\makeatletter 
\def\MT@register@subst@font{\MT@exp@one@n\MT@in@clist\font@name\MT@font@list
 \ifMT@inlist@\else\xdef\MT@font@list{\MT@font@list\font@name,}\fi}
\makeatother

\makeatletter 
\newcommand{\myitem}[1]{%
\item[(#1)]\protected@edef\@currentlabel{#1}%
}
\makeatother

\usepackage[pdftex,bookmarks,bookmarksnumbered,linktocpage,   
         colorlinks,linkcolor=blue,citecolor=blue]{hyperref}
\usepackage[capitalize,noabbrev]{cleveref} 

\newcommand{\bit}{\begin{itemize}}    
\newcommand{\eit}{\end{itemize}}
\newcommand{\ben}{\begin{enumerate}}
\newcommand{\een}{\end{enumerate}}

\newcommand{\benroman}{\ben[\normalfont (i)]}  
\let\eroman\een

\newcommand{\bde}{\begin{description}}
\newcommand{\ede}{\end{description}}
\let\oper=\mathbb                               

\newcommand{\III}{\oper{I}}                     
\newcommand{\SSS}{\oper{S}}                     
\newcommand{\VVV}{\oper{V}}                     

\theoremstyle{plain}

\newtheorem{Theorem}{Theorem}[section]

\newtheorem{Proposition}[Theorem]{Proposition}

\newtheorem{Corollary}[Theorem]{Corollary}
\newtheorem{Claim}[Theorem]{Claim}
\newtheorem{Subclaim}[Theorem]{Subclaim}

\theoremstyle{definition}

\newtheorem{exa}[Theorem]{Example}

\theoremstyle{remark}

\newtheorem{Remark}[Theorem]{Remark}

\crefname{Theorem}{Theorem}{Theorems}
\crefname{Proposition}{Proposition}{Propositions}
\crefname{Lemma}{Lemma}{Lemmas}
\crefname{Corollary}{Corollary}{Corollaries}
\crefname{Claim}{Claim}{Claims}
\crefname{Subclaim}{Subclaim}{Subclaims}
\crefname{Definition}{Definition}{Definitions}
\crefname{exa}{Example}{Examples}
\crefname{Remark}{Remark}{Remarks}
\crefname{Fact}{Fact}{Facts}
\crefname{exer}{Exercise}{Exercises}
\crefname{problem}{Problem}{Problems}

\AddToHook{env/Proposition/begin}{\crefalias{Theorem}{Proposition}}
\AddToHook{env/Lemma/begin}{\crefalias{Theorem}{Lemma}}
\AddToHook{env/Corollary/begin}{\crefalias{Theorem}{Corollary}}
\AddToHook{env/Claim/begin}{\crefalias{Theorem}{Claim}}
\AddToHook{env/Subclaim/begin}{\crefalias{Theorem}{Subclaim}}
\AddToHook{env/Definition/begin}{\crefalias{Theorem}{Definition}}
\AddToHook{env/exa/begin}{\crefalias{Theorem}{exa}}
\AddToHook{env/Remark/begin}{\crefalias{Theorem}{Remark}}

\let\leq=\leqslant

\let\geq=\geqslant 

\usepackage[mathscr]{euscript}
 \let\mathscr\relax 
\usepackage[scr]{rsfso}

\newcommand{\dom}{\mathsf{dom}}

\renewcommand{\int}{\mathsf{int}\,}

\bmdefine{\A}{A} 
\bmdefine{\C}{C}                                
                              
\bmdefine{\2}{2}
\bmdefine{\B}{B}
\bmdefine{\D}{D}
\bmdefine{\E}{E}
\bmdefine{\Luk}{L}
\bmdefine{\Term}{T} 
\bmdefine{\Free}{F}
\bmdefine{\Fb}{F}

\usepackage{scalerel}

\newcommand{\V}{\mathsf{V}}

\newcommand{\K}{\mathsf{K}}

\newcommand{\HHH}{\mathbb{H}}
\newcommand{\PPP}{\mathbb{P}}

\newcommand{\PPU}{\mathbb{P}_{\!\textsc{\textup{u}}}^{}}

\newcommand{\ext}{\mathsf{ext}}
\newcommand{\extpp}{\mathsf{ext}_{\textsc{pp}}}

\newcommand{\CMon}{\mathsf{CM}}

\newcommand{\imppp}{\mathsf{imp}_{\textsc{pp}}}

\renewcommand{\L}{\mathscr{L}}
\newcommand{\F}{\mathcal{F}}

\newcommand{\Con}{\mathsf{Con}}

\DeclareUnicodeCharacter{001B}{}



\begin{document}

\title[Implicit operations in varieties of  commutative monoids]{Implicit operations in varieties of\\  commutative monoids}

\author{Luca Carai, Miriam Kurtzhals, and Tommaso Moraschini}

\address{Luca Carai: Dipartimento di Matematica ``Federigo Enriques'', Universit\`a degli Studi di Milano, via Cesare Saldini 50, 20133 Milano, Italy}\email{luca.carai.uni@gmail.com}

\address{Miriam Kurtzhals and Tommaso Moraschini: Departament de Filosofia, Facultat de Filosofia, Universitat de Barcelona (UB), Carrer Montalegre, $6$, $08001$ Barcelona, Spain}
\email{mkurtzku7@alumnes.ub.edu \textnormal{and} tommaso.moraschini@ub.edu}
\email{ }

\begin{abstract}
    An \emph{implicit operation} of a class of similar algebras $\mathsf{K}$ is a collection of first order definable partial functions
on the members of $\mathsf{K}$ that is globally preserved by homomorphisms. For instance, ``taking inverses'' can be viewed as a unary implicit operation of the class of all monoids because its graph on a given monoid is defined by the equation $xy \thickapprox 1 \thickapprox yx$ and monoid homomorphisms preserve existing inverses.\ As this example demonstrates, the implicit operations of a class $\K$ need not be given by a term of $\K$.
We show that an equational class of commutative monoids can be expanded with enough implicit operations so that every implicit operation can be interpolated by a family of terms if and only, in each of its members, for every $a$ there exists some $b$ such that $a = a^2b$, i.e., the class consists of \emph{inverse monoids}.
Our methods build on the interaction of the theory of implicit operations with Grillet's description of finitely generated subdirectly irreducible commutative semigroups and the combinatorics deriving from an extension of Isbell's Zigzag Theorem to all equational classes of commutative monoids.
\end{abstract}

\maketitle

\section{Introduction}

\subsection*{Implicit operations} An \emph{algebraic language} is a family $\L$ of function symbols together with a map $\mathsf{ar} \colon \L \to \mathbb{N}$ that associates an arity with every member of $\L$. Then, an $\L$-\emph{algebra} is a structure $\A = \langle A; \{ f^\A : f \in \L \} \rangle$, where $A$ is a nonempty set  and $f^\A$ is a function on $A$ of arity $\mathsf{ar}(f)$ for every $f \in \L$ (see, e.g., \cite[Chap.\ 1.1]{Ber11} and \cite[Chap.\ 1.3]{ModCK}).\ In this case, we often say that $\L$ is the language of the algebra $\A$. Familiar examples of algebras include monoids, groups, rings, and Boolean algebras. Two algebras are said to be \emph{similar} when they have the same language.
 
An \emph{implicit operation} of a class of similar algebras $\K$ is a family of partial functions 
on the members of $\K$ that is globally preserved by homomorphisms and definable by a formula (see \cite[Sec.\ 3]{CKMIMP}).
More precisely, an $n$-ary \emph{operation} of $\mathsf{K}$ is a sequence $f = \langle f^\A : \A \in \mathsf{K} \rangle$, where each $f^\A \colon \dom(f^\A) \to A$ is a partial $n$-ary function on $A$ with domain $\mathsf{dom}(f^\A) \subseteq A^n$
that is  globally preserved by the homomorphisms between members of $\mathsf{K}$.
The latter means that for every homomorphism $h \colon \A \to \B$ with $\A, \B \in \K$ and $\langle a_1, \dots, a_n \rangle \in \mathsf{dom}(f^\A)$ we have
\[
\langle h(a_1), \dots, h(a_n) \rangle \in \mathsf{dom}(f^\B) \, \, \text{ and } \, \, h(f^\A(a_1, \dots, a_n)) = f^\B(h(a_1), \dots, h(a_n)).
\]
 An $n$-ary operation $f$  of $\K$ is said to be \emph{implicit} when it is defined by some first order formula $\varphi(x_1, \dots, x_n, y)$, in the sense that for all $\A \in \K$ and $a_1, \dots, a_n, b \in A$,
 \[
\A \vDash \varphi(a_1, \dots, a_n, b) \iff \langle a_1, \dots, a_n \rangle \in \mathsf{dom}(f^\A) \text{ and }f^\A(a_1, \dots, a_n) = b.
 \]
 For instance, “taking  inverses” is an implicit operation of the class of all monoids because it can be defined by the conjunction of equations $\varphi(x, y) = (xy \thickapprox 1) \sqcap (yx \thickapprox 1)$ and  monoid homomorphisms preserve  inverses when they exist.

\subsection*{Interpolation and epimorphisms} A \emph{term} of a class of similar algebras $\K$ is a formal expression obtained by applying the function symbols of $\K$ to the set of variables. For instance, $(1 \cdot (x \cdot y)) \cdot z$ is a monoid term. With every term $t(x_1, \dots, x_n)$ of $\K$ and $\A \in \K$ we associate a map $t^\A \colon A^n \to A$ that sends a tuple $\langle a_1, \dots, a_n \rangle \in A^n$ to the result of applying $t$ to the elements $a_1, \dots, a_n$.
For instance, if $t(x, y, z) = (1 \cdot (x \cdot y)) \cdot z$ and $\A$ is a monoid, then $t^\A \colon A^3 \to A$ is the map defined as $t^\A(a, b, c) = (1 \cdot (a \cdot b)) \cdot c = abc$ for all $a, b, c \in A$.

As the implicit operations of a class of similar algebras $\K$ do not need to belong to the language of $\K$, it is natural to wonder whether they can at least be interpolated by the terms of $\K$.
To this end, we say that an $n$-ary implicit operation $f$ of $\K$ is interpolated by a family of terms $\{ t_i : i \in I \}$ of $\K$ when for every $\A \in \K$ and $\langle a_1, \dots, a_n \rangle \in \mathsf{dom}(f^\A)$ there exists $i \in I$ such that $f^\A(a_1, \dots, a_n) = t_i^\A(a_1, \dots, a_n)$. Intuitively, the implicit operation $f$ is made “explicit” by the terms in $\{ t_i : i \in I \}$. If every implicit operation of $\K$ is interpolated by a family of terms of $\K$, we say that $\K$ has the \emph{strong Beth definability property} (see \cite[Sec.\ 5]{CKMIMP}).

We recall that a class of similar algebras is said to be a \emph{variety} when it can be axiomatized by equations or, equivalently, when it is closed under homomorphic images, subalgebras, and direct products (see, e.g., \cite[Thm.\ II.11.9]{BuSa00}). Notably, in the context of varieties, the strong Beth definability property is equivalent to the following purely categorical property (see \cite[Thm.\ 6.5]{CKMIMP}).
A variety $\K$ is said to have the \emph{strong epimorphism surjectivity property} when for all homomorphism $f \colon \A \to \B$ with $\A, \B \in \K$ and $b \in B - f[A]$ there exists a pair of homomorphisms $g, h \colon \B \to \C$ with $\C \in \K$ such that $g{\upharpoonright}_{f[A]} = h{\upharpoonright}_{f[A]}$ and $g(b) \ne h(b)$ (see, e.g., \cite[p.\ 176]{MakpBeth} or \cite{SurvKissal})\footnote{In \cite{SurvKissal}, the strong epimorphism surjectivity property is called IPA and defined in a different, but equivalent, way.} or, equivalently, when all  monomorphisms are regular in $\K$ (see, e.g., \cite[Prop.~6.1]{SurvKissal}). For instance, the inclusion of $\langle \mathbb{Z}; \cdot, 1 \rangle$ into $\langle \mathbb{Q}; \cdot, 1 \rangle$ witnesses  a failure of the strong  Beth definability property in the variety of all (resp.\ commutative) monoids because it shows that the implicit operation of “taking inverses” cannot be interpolated by any family of monoid terms (otherwise $\mathbb{Z}$ would be a set of generators for $\langle \mathbb{Q}; \cdot, 1 \rangle$, which is false), as well as a failure of the strong epimorphism property in the same variety because monoid homomorphisms preserve inverses.

\subsection*{The main result}

As the above example shows, the implicit operations of a class of algebras $\K$  need not be interpolated by  terms of $\K$ in general. Therefore, it is natural to expand  $\K$ by adding enough implicit operations to ensure the validity of the strong Beth definability property. The result of such an expansion of $\K$ has been called a \emph{Beth companion} of $\K$ in \cite[Sec.\ 11]{CKMIMP}. For instance, the variety of Abelian groups is a Beth companion of the class of cancellative commutative monoids (see \cite[Thm.\ 11.9(i)]{CKMIMP}) and a Beth companion for the class of reduced commutative rings can be obtained by adding the implicit operations of “taking weak inverses and weak prime roots” (see \cite{CKMRings}).
In general, a class of algebras may admit distinct Beth companions or none, the latter being the case for the class of all (resp.\ commutative) monoids (see \cite[Thm.\ 14.1]{CKMIMP}).
Nonetheless, in the context of varieties, Beth companions are essentially unique when they exist because all the Beth companions of a variety are term equivalent (see \cite[Thm.~11.7]{CKMIMP}).

Our main result is a description of the varieties of commutative monoids with a Beth companion. We recall that a commutative monoid $\langle A; \cdot, 1 \rangle$ is said to be \emph{inverse} when for every $a \in A$ there exists $b \in A$ such that $a = a^2 b$ (see, e.g., \cite[p.\ 242]{Isb65}  and \cite[Lem.~4]{RapregR}).
We also say that a variety of commutative monoids is \emph{inverse} when so are its members or, equivalently, when it validates the equation $x^{n+1} \thickapprox x$ for some $n> 0$ (see \cref{Thm : char of inverse monoids}). We will prove that a variety of commutative monoids has a Beth companion if and only if it is inverse, in which case it is its own Beth companion (see \cref{Thm : main CMon Vars}). In other words, either a variety of commutative monoids already has the strong Beth definability property (in which case it is inverse) or it is hopeless to expand it with some implicit operations to ensure the validity of this property.
As a variety is its own Beth companion precisely when it has the strong epimorphism surjectivity property (see \cite[Thms.\ 11.9(vi) and 11.6]{CKMIMP}),
we deduce that the  varieties of commutative monoids with this property are precisely the inverse ones.

We recall that the variety of all commutative monoids lacks the strong Beth definability property because, in general, the implicit operation of “taking inverses” cannot be interpolated by a family of monoid terms. On the other hand, inverses can be interpolated in every inverse variety $\K$ by a term of the form $x^{n-1}$, because $\mathsf{K}$ validates  the equation $x^{n+1} \thickapprox x$ for some $n > 0$ . It is therefore tempting to conjecture that this is the reason why the sole varieties of commutative monoids with a Beth companion are the inverse ones. However, this impression should be dispelled because every proper variety of commutative monoids $\K$ possesses an implicit operation $f$ that interpolates the implicit operation of “taking inverses” and, moreover, is \emph{extendable}, in the sense that each $\A \in \K$ can be extended to some $\B \in \K$ in which $f^\B$ is total and extends $f^\A$ (see \cref{Exa : inverses are interpolable}).
The latter makes it possible to expand $\K$ by adding $f$, thus resolving the problem of intepolating  existing inverses. However, as our theorem shows, this is not enough to ensure the validity of the strong Beth definability property, unless $\K$ was an inverse variety from the start (in which case, there is no need to expand it at all).

\subsection*{Isbell's Zigzag Theorem}

In order to explain what prevents the existence of a Beth companion for noninverse varieties of commutative monoids, we rely on the notion of a dominion. More precisely, given a subalgebra $\A$  of a member $\B$ of a variety $\K$ (in symbols $\A \leq \B \in \K$), the \emph{dominion} of $\A$ in $\B$ relative to $\K$
is the set
\begin{align*}
\mathsf{d}_\K(\A, \B) = \{ & b \in B : g(b) = h(b) \text{ for every pair of homomorphisms} \\
&g, h \colon \B \to \C  \text{ with }\C \in \K \text{ such that }g{\upharpoonright}_A = h{\upharpoonright}_A\}.
\end{align*}
As $\K$ is closed under subalgebras, $\K$ has the strong epimorphism surjectivity property if and only if $\mathsf{d}_\K(\A, \B) = A$ for every $\A \leq \B \in \K$. 

A celebrated result,  known as \emph{Isbell's Zigzag Theorem} (here, \cref{Thm : zigzag original}), provides a description of dominions in the varieties of all monoids and all (resp.\ commutative) semigroups  in terms of certain formulas, which we term \emph{Isbell's formulas} (see  \cite[Thm.\ 2.3]{Isb65},  \cite[Thm.\ 1.1]{HoIsEpiII}, and \cite[Thm.\ 1.2]{HowZigzag}). 
By means of elementary category theory tools and Head's classification of varieties of commutative monoids (see, e.g.,  \cite[Thm.\ 5.1]{GLVVCMon}), we show that Isbell's Zigzag Theorem can be extended to all such varieties (see \cref{Thm : Doms in CMon Vars}). 
For the present purpose, the interest of Isbell's Zigzag Theorem derives from the observation that the simplest  Isbell's formula 
\[
\varphi(x_1, x_2,x_3, y) = \exists w, z ((y \approx wx_2z) \sqcap (wx_2 \approx x_1) \sqcap (x_2z \approx x_3))
\]
defines an implicit operation  on every proper variety $\K$ of commutative monoids and, unless $\K$ is inverse, the proof of \cref{Thm : main CMon Vars}
shows that $\K$ cannot be expanded so that this implicit operation be interpolated by a family of terms. This impossibility proof relies on Grillet's characterization of finitely generated subdirectly irreducible commutative semigroups (see \cite[Cor.\ IV.4.6]{Gri01})
and requires the full power of Isbell's Zigzag Theorem for $\K$.

\section{Commutative monoids}

\subsection{Basic concepts} 

Given a binary operation  $\cdot$ on a set $A$ and $a, b \in A$, we will write $ab$ as a shorthand for $a \cdot b$. Moreover, we recall that
\benroman
\item a \emph{semigroup} is an algebra $\langle A; \cdot \rangle$, where $\cdot$ is a binary associative operation on $A$;
\item an element $a$ of a semigroup $\langle A; \cdot \rangle$ is  \emph{neutral} when $ab = b = ba$ for every $b \in A$;
\item a \emph{monoid} is an algebra $\langle A; \cdot, 1 \rangle$, where $\langle A; \cdot \rangle$ is a semigroup with neutral element $1$.
\eroman
A semigroup may have at most one neutral element.\
A monoid or a semigroup is said to be \emph{commutative} when the operation $\cdot$ is commutative. The \emph{semigroup reduct} of a monoid $\A = \langle A; \cdot, 1 \rangle$ is the semigroup $\A_\mathsf{s} = \langle A; \cdot \rangle$.

An element $a$ of a commutative semigroup $\A$ is said to be
    \benroman
    \item a \emph{zero element} when $ab = a$ for every $b \in A$;
    \item $n$-\emph{nilpotent} for some $n \in \mathbb{N}$ when $a^n$ is a zero element;
    \item \emph{nilpotent} when it is $n$-nilpotent for some $n \in \mathbb{N}$; 
    \item \emph{cancellative} when $ab = ac$ implies that $b = c$ for all $b, c \in A$; 
    \item an \emph{inverse} of some $b \in A$ when $ab$ is neutral, in which case we write $a^{-1} = b$; 
    \item \emph{invertible} when it possesses an inverse.
    \eroman
A commutative semigroup or monoid $\A$ is called \emph{nilpotent} (resp.\ 
\emph{cancellative}) when so are all its elements. 

The following is an immediate consequence of the definitions. 

\begin{Proposition}\label{Prop : basic properties}
The following conditions hold for all commutative semigroups $\A$ and $a \in A$:
\benroman
\item\label{item : basic properties : 1} $\A$ may have at most one zero element;
\item\label{item : basic properties : 2} $a$ may have at most one inverse;
\item\label{item : basic properties : 4} if $a$ is invertible, then it is cancellative.
\eroman
\end{Proposition}

\subsection{Varieties}

Let $\III,  
\HHH, \SSS, \PPP$, and $\PPU$
be the class operators of closure under 
isomorphisms, 
homomorphic images, subalgebras, direct products, and ultraproducts.
Then, a class of similar algebras is said to be a 
\emph{variety} when it is closed under $\HHH, \SSS$, and $\PPP$ or, equivalenlty, when it can be axiomatized by a set of equations (see, e.g., \cite[Thm.\ II.11.9]{BuSa00}  due to Birkhoff).
For every class of similar algebras $\K$ there exists the least 
variety containing $\K$, which we denote by 
$\VVV(\K)$.

We consider the following classes of algebras for every $m, n \in \mathbb{N}$:
\begin{align*}
\mathsf{S} &= \text{the variety of semigroups};\\ 
\mathsf{CS} &= \text{the variety of commutative semigroups};\\ 
\mathsf{M} &= \text{the variety of monoids};\\ 
\mathsf{CM} &= \text{the variety of commutative monoids};\\
\mathsf{A}_n &= \text{the variety of commutative monoids satisfying }x^n \thickapprox 1;\\
\mathsf{C}_n &= \text{the variety of commutative monoids satisfying }x^n \thickapprox x^{n+1};\\
\mathsf{V}_{m, n} &=\text{the variety of commutative monoids satisfying }x^{m+n} \thickapprox x^{n}.
\end{align*}

\begin{Remark}
For every $n > 0$ the members of $\mathsf{A}_n$ can be viewed as groups because for all $\A \in \mathsf{A}_n$ and $a \in A$ we have $a^{-1} = a^{n-1}$.
\qed
\end{Remark}

\begin{Proposition}\label{Prop : variety equalities}
For every $m, n \in \mathbb{N}$,
\begin{align*}
\mathsf{A}_0 &= \mathsf{CM};\\
\mathsf{A}_1 &= \mathsf{C}_0 = \text{the variety of trivial monoids};\\
\mathsf{V}_{m, n} &= \VVV(\mathsf{A}_m \cup \mathsf{C}_n).
\end{align*}
Furthermore, $\mathsf{V}_{m, n}$ is a proper subvariety of $\mathsf{CM}$ if and only if $m > 0$.
\end{Proposition}

\begin{proof}
The first 
two
equalities are immediate consequences of the definitions. 
The last equality is established as follows. 
When $m=0$, using the definition of $\V_{0, n}$ and the first equality in the above display, we obtain $\mathsf{V}_{0, n} =\mathsf{CM} = \VVV(\mathsf{A}_0 \cup \mathsf{C}_n)$. So, we may assume $m>0$.
From \cite[Thm.~5.1]{GLVVCMon} it follows that $\mathsf{V}_{m, n} = \VVV(\mathsf{A}_s \cup \mathsf{C}_t)$ for some $s,t \in \mathbb{N}$. Suppose, with a view to contradiction, that $s$ does not divide $m$, and let $\A$ be the monoid reduct of the cyclic group of order $s$. Then $\A \in \VVV(\mathsf{A}_s \cup \mathsf{C}_t) - \mathsf{V}_{m, n}$, which contradicts $\mathsf{V}_{m, n} = \VVV(\mathsf{A}_s \cup \mathsf{C}_t)$. Therefore, $s$ divides $m$. Assume, again with a view to contradiction, that $n < t$. Consider the monoid $\B=\langle \mathbb{N}; + , 0\rangle/\theta$, where $\theta$ is the congruence defined by $\langle a,b \rangle \in \theta$ if and only if  $a=b$ or $t \leq  a,b$.
Then $\B \in \VVV(\mathsf{A}_s \cup \mathsf{C}_t) - \mathsf{V}_{m, n}$, which contradicts $\mathsf{V}_{m, n} = \VVV(\mathsf{A}_s \cup \mathsf{C}_t)$. 
Therefore, $s$ divides $m$ and $t \leq n$.
From the definitions of the varieties involved, it follows that $\V_{m,n} = \VVV(\mathsf{A}_s \cup \mathsf{C}_t) \subseteq \VVV(\mathsf{A}_m \cup \mathsf{C}_n)$.
Since it is immediate 
to see that $\VVV(\mathsf{A}_m \cup \mathsf{C}_n) \subseteq \mathsf{V}_{m, n}$, we conclude that $\mathsf{V}_{m, n} = \VVV(\mathsf{A}_m \cup \mathsf{C}_n)$.

It only remains to prove the last part of the statement. Observe that $\mathsf{V}_{m, n}$ is axiomatized by $x^{m+n}\thickapprox x^n$, which is not valid in $\langle \mathbb{N}; + , 0\rangle$ when $m>0$. Therefore, $\mathsf{V}_{m, n}$ is proper if and only if $m > 0$.
\end{proof}

Varieties of commutative monoids admit a transparent classification due to Head.
 More precisely, every subvariety of $\CMon$ is of the form $\VVV(\mathsf{A}_m \cup \mathsf{C}_n)$  (see, e.g.,  \cite[Thm.\ 5.1]{GLVVCMon}). Together with \cref{Prop : variety equalities}, this yields the following.

\begin{Theorem} \label{Thm : classification of CMon-Vars}
    A class $\K$ of commutative monoids is a variety if and only if $\K = \mathsf{V}_{m, n}$ for some $ m, n \in \mathbb{N}$. 
\end{Theorem}

\subsection{Subdirect irreducibility} 

When ordered under inclusion, the set of congruences of an algebra $\A$ forms an algebraic lattice, denoted by  $\mathsf{Con}(\A)$, whose minimum is the identity relation $\mathsf{id}_A$ on $A$  (see, e.g., \cite[Thm.\ II.5.5]{BuSa00}). When $\mathsf{id}_A$ is completely meet irreducible in $\Con(\A)$, we say that $\A$ is \emph{subdirectly irreducible}. Given a class of algebras $\K$, we let
\[
\K_\textsc{si} = \{ \A \in \K : \A \text{ is subdirectly irreducible} \}.
\]
The importance of subdirectly irreducible algebras derived from Birkhoff’s Subdirect Decomposition Theorem,
which implies that $\K = \III\SSS\PPP(\K_\textsc{si})$ for every variety $\K$ (see, e.g., \cite[Thm.~II.8.6]{BuSa00}). 

Given an algebra $\A$ and $X \subseteq A$, we denote by $\mathsf{Sg}^\A(X)$ the subuniverse of $\A$ generated by $X$. 
We say that $\A$ is \emph{finitely generated} when $A = \mathsf{Sg}^\A(X)$ for some finite $X \subseteq A$. Lastly, a \emph{universal formula} is an expression of the form $\forall x_1, \dots, x_n \varphi$, where $\varphi$ is a quantifier-free formula.
\begin{Proposition}\label{Prop : universal sentences}
Let $\K$ be a variety and
\[
\K^* = \{ \A \in \K_\textsc{si} : \A \text{ is finitely generated} \}.
\]
Every universal formula valid in $\K^*$ is also valid in $\K_\textsc{si}$.
\end{Proposition}

\begin{proof}
From  \cite[Prop.~2.16]{CKMIMP} it follows that $\K = \III \SSS \PPP \PPU(\K^*)$, and \cite[Lem.~1.5]{CD90} implies that $(\III \SSS \PPP \PPU(\K^*))_\textsc{si}\subseteq \III\SSS\PPU(\K^*)$. Thus, $\K_\textsc{si} \subseteq \III\SSS\PPU(\K^*)$. As the validity of universal formulas is preserved by $\III, \SSS$, and $\PPU$
(see, e.g., \cite[Thm.~V.2.20]{BuSa00}), we are done.
\end{proof}

Recall from \cref{Thm : classification of CMon-Vars} that every variety of commutative monoids is of the form $\V_{m, n}$ for some $m, n \in \mathbb{N}$. Furthermore, $\V_{0, n} = \mathsf{CM}$ for every $n \in \mathbb{N}$ by \cref{Prop : variety equalities}. Therefore, the next result applies to all the proper subvarieties of $\mathsf{CM}$, that is, the varieties of the form $\V_{m, n}$ with $m > 0$ (see \cref{Prop : variety equalities}).

\begin{Theorem} \label{Thm : char of SIs: inv or nil}
    Let $\A \in \V_{m,n}$ be subdirectly irreducible for some $m, n \in \mathbb{N}$ with $m > 0$. Then every $a \in A$ is either $n$-nilpotent or invertible with $a^{-1} = a^{m-1}$.
\end{Theorem}

In order to prove \cref{Thm : char of SIs: inv or nil}, we need
the following fact about the finitely generated members of $\CMon_\textsc{si}$.

\begin{Theorem}\label{Thm : Grillet finite SI}
Let $\A \in \CMon_\textsc{si}$ be finitely generated. Then $\A$ is finite and every element of $\A$ is either nilpotent or invertible.
\end{Theorem}

\begin{proof}
Let $\A_\mathsf{s}$ be the semigroup reduct of $\A$. Consider a finite set $X$ of generators of $\A$. Then $\A_\mathsf{s}$ is generated by $X \cup \{1\}$. Moreover, observe that $\Con(\A) = \Con(\A_\mathsf{s})$. Together with the assumption that $\A$ is subdirectly irreducible, this ensures that so is $\A_\mathsf{s}$. 
Consequently, $\A_\mathsf{s}$ is finitely generated and subdirectly irreducible. Hence, $\A$ is finite by \cite{Mal58} (see also \cite[Cor.~3.9]{Gri77}). Since $\A_\mathsf{s}$ is finite, \cite[Cor.\ IV.4.6]{Gri01} implies that every $a \in A$ is either nilpotent or invertible.
\end{proof}

We are now ready to prove \cref{Thm : char of SIs: inv or nil}.

\begin{proof}
We begin by proving the following.

\begin{Claim}\label{Claim : fin SI members}
    Let $\B \in \V_{m,n}$ be finitely generated and subdirectly irreducible for some $m, n \in \mathbb{N}$ with $m > 0$. Then every $b \in B$ is either $n$-nilpotent or invertible with $b^{-1} = b^{m-1}$.
\end{Claim}

\begin{proof}[Proof of the Claim]
As $\B$ is a finitely generated member of $\CMon_\textsc{si}$, every element of $\B$ is either nilpotent or invertible by \cref{Thm : Grillet finite SI}. 
 Consider $b \in B$. We have two cases depending of whether $b$ is nilpotent or invertible. First, suppose that $b$ is nilpotent.  Then there exists $k \in \mathbb{N}$ such that $b^k$ is a zero element of $\B$.
We will show that $b$ is $n$-nilpotent. If $k \leq n$, this follows immediately from the assumption that $b$ is $k$-nilpotent.
If $k > n$, let
\[
t = \min \{ r \in \mathbb{N} : k \leq n+rm \},
\]
which exists because $m > 0$ by assumption. Then
\[
b^n = b^{n+rm} = b^{n + rm-k}b^k = b^k,
\]
where the first equality holds because $\B \in \V_{m, n}$ by assumption and $\V_{m, n} \vDash x^{m+n} \thickapprox x^n$ by definition, the second is straightforward, and the third holds because $b^k$ is a zero element. Consequently, $b^n$ is a zero element, whence $b$ is $n$-nilpotent.

Next, we consider the case where $b$ is invertible. As $\B \in \V_{m, n}$ and $\V_{m, n} \vDash x^{m+n} \thickapprox x^n$, we have $b^{m+n} = b^n = 1b^n$. Since $b$ is cancellative by \cref{Prop : basic properties}(\ref{item : basic properties : 4}), we obtain $b^m = 1$. 
Together with $m > 0$, this yields that $b$ is invertible with $b^{-1} = b^{m-1}$.
\end{proof}

From \cref{Claim : fin SI members} it follows that the universal formula
\[
\varphi = \forall x,y\, ((x^ny \approx x^n) \sqcup (x^m \approx 1))
\]
is valid in the finitely generated subdirectly irreducible members of $\V_{m, n}$. Therefore, $\varphi$ holds in every subdirectly irreducible member of $\V_{m, n}$ by \cref{Prop : universal sentences}. As $\A$ is a subdirectly irreducible member of $\V_{m, n}$ by assumption, we conclude that every $a \in A$ is either $n$-nilpotent or invertible with $a^{-1} = a^{m-1}$. 
\end{proof}

Lastly, we will make use of the next observation.

\begin{Proposition} \label{Prop : zeros preserved in extensions}
Let $\A \in \CMon_{\textsc{si}}$ and $a \in A$ such that $a=a^2$. Then   $a=1$ or $a$ is a zero element of $\A$.
\end{Proposition}

\begin{proof}
Suppose first that $\A$ is finitely generated. \cref{Thm : Grillet finite SI} implies that $a$ is invertible or nilpotent. If $a$ is invertible, then $a=a^2 a^{-1} = a a^{-1} = 1$, where the second equality holds because $a=a^2$ and the others are straightforward. If $a$ is nilpotent, then $a^n$ is a zero element for some $n >0$. Since $a=a^2$, we obtain that $a=a^n$, and so $a$ is a zero element. Therefore, the statement holds for every finitely generated member of $\mathsf{CM}_\textsc{si}$. Together with \cref{Prop : universal sentences}, this ensures that the universal formula
\[
\varphi = \forall x,y\, (x\approx x^2 \to (x\approx 1 \sqcup xy\approx x))
\]
is valid in every member of $\CMon_{\textsc{si}}$. Thus, for all $\A \in \CMon_{\textsc{si}}$ and $a \in A$ such that $a=a^2$, the validity of $\varphi$ implies that  $a=1$ or $a$ is a zero element of $\A$.
\end{proof}

\subsection{Inverse  monoids}

We recall that a monoid (resp.\ semigroup) $\A$ is said to be \emph{inverse} when for every $a \in A$ there exists a unique $b \in A$ such that
\[
a = aba\, \, \text{ and } \, \, b = bab
\]
(see, e.g., \cite[p.\ 242]{Isb65}).
We say that a  variety of monoids is \emph{inverse} when so are its members. Under the assumption of commutativity, the definition of an inverse monoid can be simplified as follows (see, e.g., \cite[p.\ 242]{Isb65}  and \cite[Lem.~4]{RapregR}).
    \begin{Proposition} \label{Prop : inverse condition for CMons}
        A commutative monoid $\A$ is  inverse if and only if for every $a \in A$ there exists $b \in A$ such that $a = a^2b$.
    \end{Proposition}

    The inverse varieties of commutative monoids admit the following transparent description.

    \begin{Theorem} \label{Thm : char of inverse monoids}
        The following conditions are equivalent for a variety $\K$ of commutative monoids:
        \benroman
        \item\label{item : inverse : 1} $\K$ is inverse;
        \item\label{item : inverse : 2} $\K$ does not contain $\mathsf{C}_2$;
        \item\label{item : inverse : 3} $\K \vDash x^{n+1} \thickapprox x$ for some $n > 0$.
        \eroman
    \end{Theorem}

\begin{proof}
(\ref{item : inverse : 1})$\Rightarrow$(\ref{item : inverse : 2}): Suppose, with a view to contradiction, that $\K$ is inverse and contains $\mathsf{C}_2$.
Then consider the unique  commutative monoid $\A$ such that $A = \{ 1, a, b\}$, $b$ is a zero element, and $b = a^2$. 
As $\A$ satisfies the equation $x^{2} \thickapprox x^3$, it belongs to $\mathsf{C}_2$. Together with the assumption that $\mathsf{C}_2 \subseteq \K$, this yields $\A \in \K$. Therefore, $\A$ is inverse. By \cref{Prop : inverse condition for CMons} there exists $c \in A$ such that $a = a^2 c$. Since $b = a^2$ is a zero element, this yields $a = bc = b$, which is false. 

   (\ref{item : inverse : 2})$\Leftrightarrow$(\ref{item : inverse : 3}):  From the description of the lattice of varieties of commutative monoids (see e.g., \cite[Thm.\ 5.1]{GLVVCMon})
   it follows that for every variety $\K \subseteq \mathsf{CM}$ we have  $\mathsf{C}_2 \nsubseteq \K$  if and only if  $\K \subseteq \V_{n,1}$ for some $n>0$, but the latter condition is equivalent to $\K \vDash x^{n+1} \thickapprox x$ for some $n > 0$ by the definition of $\V_{n, 1}$.
  
(\ref{item : inverse : 3})$\Rightarrow$(\ref{item : inverse : 1}): Consider $\A \in \mathsf{K}$ and $a \in A$. Since $a^{n+1} = a$ by  \eqref{item : inverse : 3}, we obtain $a = a^2b$ for $b = a^{n-1}$. Hence, $\A$ is inverse  by \cref{Prop : inverse condition for CMons}, and so is $\mathsf{K}$.
\end{proof}

The next result of Howie and Isbell will be used later (see \cite[Thm.\ 2.3]{HoIsEpiII}).

\begin{Theorem}\label{Thm : inverse semigroups : absolutely closed}
For all $\A \leq \B \in \mathsf{S}$ with $\A$ inverse we have $\mathsf{d}_\mathsf{S}(\A, \B) = A$.
\end{Theorem}

\section{Isbell's Zigzag Theorem}

For each $n \geq 1$ let
\[
\psi_n(z_1, \dots, z_n, w_1, \dots, w_n, x_1, \dots, x_{2n+1}, y)
\]
be the conjunction of the following equations in the language of semigroups:
\begin{align}
y &\thickapprox x_1 z_1;\label{Isbell : 1 : formula}\\
x_1 & \thickapprox w_1 x_2;\label{Isbell : 2 : formula}\\
x_{2i}z_i & \thickapprox x_{2i+1}z_{i+1} \text{ for }i = 1, \dots, n-1;\label{Isbell : 3 : formula}\\
w_i x_{2i+1} & \thickapprox w_{i+1}x_{2(i+1)} \text{ for }i = 1, \dots, n-1;\label{Isbell : 4 : formula}\\
x_{2n}z_n & \thickapprox x_{2n+1};\label{Isbell : 5 : formula}\\
w_n x_{2n+1}& \thickapprox y.
\label{Isbell : 6 : formula}
\end{align}
Then let $\varphi_0(x,y) = \psi_0(x, y) = x \thickapprox y$ and for each $n \geq 1$,
\[
\varphi_n(x_1, \dots, x_{2n+1}, y) = \exists z_1, \dots, z_n, w_1, \dots, w_n \psi_n(z_1, \dots, z_n, w_1, \dots, w_n, x_1, \dots, x_{2n+1}, y).
\]
We refer to $\varphi_n$ as to the $n$-th \emph{Isbell's formula}.

The following description of dominions is known as Isbell's Zigzag Theorem (see \cite[Thm.\ 2.3]{Isb65},  \cite[Thm.\ 1.1]{HoIsEpiII}, and \cite[Thm.\ 1.2]{HowZigzag}).

\begin{Theorem} \label{Thm : zigzag original}
Let $\mathsf{K} \in \{ \mathsf{S}, \mathsf{CS}, \mathsf{M} \}$, $\A \leq \B \in \mathsf{K}$, and $b \in B$. Then $b \in \mathsf{d}_{\mathsf{K}}(\A, \B)$ if and only if
    there exist $n \in \mathbb{N}$ and $a_1, \dots, a_{2n+1} \in A$ such that $\B \vDash \varphi_n(a_1, \dots, a_{2n+1}, b)$. 
\end{Theorem}

The aim of this section is to extend Isbell's Zigzag Theorem to all varieties of commutative monoids. More precisely, we will prove the following.

\begin{Theorem} \label{Thm : Doms in CMon Vars}
Let $\mathsf{K}$ be a variety of commutative monoids, $\A \leq \B \in \mathsf{K}$, and $b \in B$. Then $b \in \mathsf{d}_{\mathsf{K}}(\A, \B)$ if and only if
    there exist $n \in \mathbb{N}$ and $a_1, \dots, a_{2n+1} \in A$ such that $\B \vDash \varphi_n(a_1, \dots, a_{2n+1}, b)$. 
\end{Theorem}

To this end, we will make use of some elementary category theory, for which we refer to  \cite{MR1712872}. As a first step, we recall that every variety $\K$ can be viewed as a category whose objects are the members of $\K$ and whose arrows are the homomorphisms between them. As such, $\K$ is a bicomplete category  and, in particular, closed under small colimits (see, e.g., \cite[Prop.~9.4.8 \& Thm.~9.4.14]{Ber15}).\footnote{From the perspective of category theory, it is convenient to add an empty algebra to varieties whose languages do not comprise any constant symbol. This issue, however, is immaterial for varieties of monoids.}
Therefore, given $\A \leq \B \in \K$, we denote by $p_1, p_2 \colon \B \to \B \ast ^{\mathsf{K}}_\A \B$ the pair of arrows of $\K$ consisting of the pushout in $\K$ of the diagram given by two copies of the inclusion map from  $\A$ to $\B$. 
 Dominions and pushouts are connected as follows (see \cite[Lem.1.1]{Isb65}).
\begin{Theorem}\label{Thm : dominions = pushouts}
    Let $\mathsf{K}$ be a variety, $\A \leq \B \in \mathsf{K}$, and $p_1, p_2 \colon \B \to \B \ast ^{\mathsf{K}}_\A \B$ the pushout in $\K$ of two copies of the inclusion map from $\A$ to $\B$. Then
    \[
    \mathsf{d}_{\mathsf{K}}(\A, \B) = \{ b \in B : p_1(b) = p_2(b) \}.
    \]
\end{Theorem}

Let $\V$ be a subvariety of a variety $\K$. We say that $\V$ is \emph{closed under $\K$-pushouts} when the pushout in $\V$ of each pair of arrows $g \colon \A \to \B$ and $h \colon \A \to \C$ of $\V$ coincides with the pushout of $g$ and $h$ in $\K$. Similarly, $\V$ is said to be  \emph{closed under  $\K$-colimits} when the colimit in $\V$ of every small diagram in $\V$ coincides with the colimit in $\K$ of the same small diagram. As pushouts are a special kind of  small colimits, closure under  $\K$-colimits implies closure under $\K$-pushouts. 

\begin{Corollary} \label{Cor : closure under pushouts implies same doms}
    Let $\V$ be a subvariety of a variety $\K$ and assume that $\V$ is closed under $\K$-pushouts. The for all $\A \leq \B \in \V$ we have $\mathsf{d}_{\mathsf{V}}(\A, \B) = \mathsf{d}_{\mathsf{K}}(\A, \B)$. 
\end{Corollary}

\begin{proof}
    Immediate from  \cref{Thm : dominions = pushouts}.
\end{proof}

Let $\K$ be a variety of commutative monoids. In view of \cref{Thm : classification of CMon-Vars}, $\K$ is axiomatized relative to $\mathsf{CM}$ by $x^{m+n} \thickapprox x^n$ for some $m, n \in \mathbb{N}$. Observe that  for every $\A \in \mathsf{CM}$ the set
\[
\{ a \in A : a^{m+n} = a^n \}
\]
is the universe of a subalgebra $\Phi_\K(\A)$ of $\A$ that, moreover, belongs to $\K$. In addition, for every homomorphism $h \colon \A \to \B$ between members of $\mathsf{CM}$ the map $\Phi_\K(h) \colon \Phi_\K(\A) \to \Phi_\K(\B)$ defined as $\Phi_\K(h)(a) = h(a)$ is also a homomorphism. Therefore, $\Phi_\K \colon \mathsf{CM} \to \K$ can be viewed as a covariant functor (for a generalization of this construction, see \cite[Thm.\ 2.8]{MoAdj}). Lastly, let $\Psi_\K \colon \K \to \mathsf{CM}$ be the inclusion functor. 
The next result establishes that $\K$ is a coreflective subcategory of $\CMon$, that is, $\Psi_\K$ is a left adjoint.

\begin{Theorem}\label{Thm : McKenzie adjunctions}
For every variety $\K$ of commutative monoids  $\Psi_\K$ is left adjoint to $\Phi_\K$.
\end{Theorem}

\begin{proof}
    In view of \cite[Thm.\ IV.2(ii)]{MR1712872}, it suffices to show that for every homomorphism $h \colon \A \to \B$ with  $\A \in \K$ and $\B \in \mathsf{CM}$ we have $h[A] \subseteq \Phi_\K(\B)$. To this end, consider $h \colon \A \to \B$ as above. Since $\K$ is a variety of commutative monoids, it is axiomatized relative to $\mathsf{CM}$ by $x^{m+n} \thickapprox x^n$ for some $m, n \in \mathbb{N}$ (see \cref{Thm : classification of CMon-Vars}). Then consider $a \in A$. As $\A \in \K$, we have $a^{m+n} = a^n$. Since $h$ is a homomorphism, this yields $h(a)^{m+n} = h(a)^n$, whence $h(a) \in \Phi_\K(\B)$ by the definition of $\Phi_\K(\B)$.
\end{proof}

\begin{Corollary}
    \label{Cor : 1-var axioms implies closed under pushouts}
    Every variety of commutative monoids is closed under  $\mathsf{CM}$-colimits.
\end{Corollary} 

\begin{proof}
Let $\K$ be a variety of commutative monoids and $\mathsf{D}$ a small diagram in $\K$. Moreover, let $X$ be the colimit of $\mathsf{D}$ in $\K$. Since $\Psi_\K \colon \K \to \mathsf{CM}$ is a left adjoint functor by \cref{Thm : McKenzie adjunctions} and  left adjoint functors preserve small colimits, $\Psi_\K(X)$ is the colimit of $\Psi_\K(\mathsf{D})$ in $\mathsf{CM}$. As $\Psi_\K$ is the inclusion functor of $\K$ into $\mathsf{CM}$, we have $X = \Psi_\K(X)$ and $\mathsf{D} = \Psi_\K(\mathsf{D})$. Hence $X$ is the colimit of $\mathsf{D}$ in $\mathsf{CM}$ as well.
\end{proof}

From Corollaries \ref{Cor : closure under pushouts implies same doms} and \ref{Cor : 1-var axioms implies closed under pushouts} we deduce the following.

\begin{Corollary} \label{Cor : compute doms of subvarieties in CMon}
Let $\mathsf{K}$ be a variety of commutative monoids. Then for all $\A \leq \B \in \K$,
\[
\mathsf{d}_{\mathsf{K}}(\A, \B) = \mathsf{d}_{\mathsf{CM}}(\A, \B).
\]
\end{Corollary}

Recall that the semigroup reduct of a monoid $\A$ is denoted by $\A_\mathsf{s}$. In view of \cref{Thm : zigzag original} and \cref{Cor : compute doms of subvarieties in CMon}, to prove \cref{Thm : Doms in CMon Vars}, it will be enough to establish the following.

\begin{Proposition} \label{Prop : dom computed in CSG vs CMon}
For all $\A \leq \B \in \mathsf{CM}$,
\[
\mathsf{d}_{\mathsf{CM}}(\A, \B) =  \mathsf{d}_{\mathsf{CS}}(\A_\mathsf{s},  \B_\mathsf{s}).
\]
\end{Proposition}
\begin{proof}
To prove the inclusion from right to left, we reason by contraposition. Consider $b \in B -  \mathsf{d}_{\mathsf{CM}}(\A, \B)$. Then there exists a pair of monoid homomorphisms $g,h \colon \B \to \C$ with  $\C \in \mathsf{CM}$ such that $g {\upharpoonright}_A = h {\upharpoonright}_A$ and $g(b) \neq h(b)$.
Then  $g,h \colon \B_\mathsf{s} \to \C_\mathsf{s}$ is a pair of semigroup homomorphisms with $\C_\mathsf{s} \in \mathsf{CS}$ witnessing  $b \notin \mathsf{d}_{\mathsf{CS}}(\A_\mathsf{s},  \B_\mathsf{s})$, as desired.

Next, we prove the inclusion from left to right. 
Again, we proceed by contraposition. Consider  $b \in B -  \mathsf{d}_{\mathsf{CS}}(\A_\mathsf{s}, \B_\mathsf{s})$. Then there exists a pair of semigroup homomorphisms $g,h \colon \B_\mathsf{s} \to \C$ with $\C \in \mathsf{CS}$ such that $g {\upharpoonright}_A = h {\upharpoonright}_A$ and $g(b) \neq h(b)$. We may assume that $C = \mathsf{Sg}^{\C}(g[B] \cup h[B])$, otherwise we replace $\C$ by its subalgebra with universe $\mathsf{Sg}^{\C}(g[B] \cup h[B])$.

\begin{Claim}\label{Claim : C is a monoid really}
The element $g(1^{\B})$ is neutral in $\C$ and $g(1^{\B}) = h(1^\B)$.
\end{Claim}

\noindent \textit{Proof of the Claim.}
First, observe that $1^{\B} \in A$ because $\A$ is a subalgebra of the monoid $\B$. Together with the assumption that $g {\upharpoonright}_A = h {\upharpoonright}_A$, this yields $g(1^{\B}) = h(1^{\B})$.

    Next, we prove that $g(1^\B)$ is neutral in $\C$. Consider an element $c \in C$. As $\C$ is commutative, it will be enough to show that  $c = g(1^{\B})c$. Since $c \in C = \mathsf{Sg}^\C(g[B]\cup h[B])$ and  $g[B]$ and $h[B]$ are both universes of subalgebras of $\C$, there exist $a_1, a_2 \in B$ such that
\[
\text{either} \quad c = g(a_1),\quad \text{or} \quad c = g(a_1)h(a_2), \quad \text{or} \quad c = h(a_2).
\]
Observe that $g(1^{\B})g(a_1) = g(1^{\B}a_1) = g(a_1)$. When combined with any of the first two cases in the above display, this yields $g(1^\B)c = c$, as desired. Therefore, it only remains to consider the case where $c = h(a_2)$. Recall that $g(1^\B) = h(1^\B)$. 
    Hence,     
    \[
    \pushQED{\qed}g(1^{\B})c = h(1^{\B})c = h(1^\B)h(a_2) = h(1^\B a_2) = h(a_2) = c.\qedhere \popQED\] 

By \cref{Claim : C is a monoid really} the commutative semigroup $\C$ can be viewed as a commutative monoid $\C^*$, whose neutral element is $g(1^\B) = h(1^\B)$. Consequently, the semigroup homomorphisms $g, h \colon \B_\mathsf{s} \to \C$ can be viewed as monoid homomorphisms $g, h \colon \B \to \C^*$. As such, $g$ and $h$ witness $b \notin \mathsf{d}_{\mathsf{CM}}(\A, \B)$ because $g {\upharpoonright}_A = h {\upharpoonright}_A$ and $g(b) \neq h(b)$ by assumption.
\end{proof}

\section{Isbell's formulas}

In this section, we establish some properties of Isbell's formulas that will be needed later on. Recall that $\psi_n$ is the conjunction of the formulas in (\ref{Isbell : 1 : formula})--(\ref{Isbell : 6 : formula}).

\begin{Proposition} \label{Prop : Isb argument 0 -> y = 0}
        Let $\A \in \mathsf{CM}$ and 
        $c_1, \dots, c_n, d_1, \dots, d_n, a_1, \dots, a_{2n+1}, b \in A$  with $n>0$ be such that 
        \[\A \vDash \psi_n(c_1, \dots, c_n, d_1, \dots, d_n, a_1, \dots, a_{2n+1}, b). \]
        Moreover, let $d_0 = 1$. 
        Then for every $0 \leq m \leq n-1$ we have  
        \[
        d_{m}a_{2m+1}c_{m+1} = b = d_{m+1}a_{2(m+1)}c_{m+1}.
        \]
    \end{Proposition}

    \begin{proof}
        We proceed by induction on $m$. 
        For $m = 0$, we have to show that $d_0a_1c_1 = b = d_1a_2c_1$. 
        Since $d_0 = 1$ by definition, this amounts to $a_1 c_1 = b = d_1a_2c_1$. Observe that $b = a_1c_1$ by~(\ref{Isbell : 1 : formula}). Moreover, by (\ref{Isbell : 2 : formula}) we have $a_1 = d_1a_2$. Together with $b= a_1c_1$, this yields $b = d_1a_2 c_1$, as desired.
        
        For the induction step, suppose that the statement holds for $0 \leq m < n-1$. We will show that it holds for $m+1$ as well. To this end, observe that (\ref{Isbell : 3 : formula}) and (\ref{Isbell : 4 : formula}) give
        \begin{equation}\label{Eq : avoid overfull}
                    a_{2(m+1)}c_{m+1} = a_{2(m+1)+1}c_{m+2} \, \, \text{ and }\, \, d_{m+1}a_{2(m+1)+1} = d_{m+2}a_{2(m+2)}.
        \end{equation}
From the right-hand side of (\ref{Eq : avoid overfull}) it follows that $d_{m+1}a_{2(m+1)+1}c_{m+2} = d_{m+2}a_{2(m+2)}c_{m+2}$. Therefore, it only remains to show that $d_{m+1}a_{2(m+1)+1}c_{m+2} = b$. Observe that, applying in succession the left-hand side of (\ref{Eq : avoid overfull}) and the inductive hypothesis, we obtain $d_{m+1}a_{2(m+1)+1}c_{m+2} = d_{m+1}a_{2(m+1)}c_{m+1} = b$, as desired.
\end{proof}

Recall that $\varphi_n$ is the $n$-th Isbell's formula.

    \begin{Corollary} \label{Cor: : arg in Isb = 0 - > value 0}
    Let $\A \in \mathsf{CM}$  with a zero element $0$ and $a_1, \dots, a_{2n+1}, b \in A$ such that $\A \vDash  \varphi_n(a_1, \dots, a_{2n+1}, b)$. If $a_m = 0$ for some $1 \leq m \leq 2n+1$, then $b = 0$. 
    \end{Corollary}
    
    \begin{proof}
     If $n=0$, then $a_1=b$ by the definition of $\varphi_0$. Hence, $a_1=0$ implies $b=0$, as desired. Then suppose that $n>0$.
      From $\A \vDash  \varphi_n(a_1, \dots, a_{2n+1}, b)$ and the definition of $\varphi_n$ it follows that 
      \[
      \A \vDash \psi_n(c_1, \dots, c_n, d_1, \dots, d_n, a_1, \dots, a_{2n+1}, b)
      \]
      for some $c_1, \dots, c_n, d_1, \dots, d_n \in A$. As $1 \leq m \leq 2n+1$ by assumption, either  $m=2n+1$ or $m \in \{2k+1,2(k+1)\}$ for some $0 \leq k \leq n-1$. First, suppose that $m = 2n+1$. Then (\ref{Isbell : 6 : formula}) gives $d_n a_{2n+1}= b$. As $a_{2n+1} = a_m = 0$ by assumption and $0$ is a zero element for $\A$, this yields $b = 0$, as desired. Next, we consider the case where $m = 2k+1$ (resp.\ $m = 2(k+1)$)  for some $0 \leq k \leq n-1$. By Proposition \ref{Prop : Isb argument 0 -> y = 0} and the assumption that $a_m$ is a zero element for $\A$ we obtain $b = d_{k}a_{2k+1}c_{k+1} = d_k a_m c_{k+1} = 0$ (resp.\ $b = d_{k+1}a_{2(k+1)}c_{k+1} = d_{k+1}a_{m}c_{k+1} = 0$).
    \end{proof}

        We will make use of the following observations.

\begin{Proposition}\label{prop:zigzag facts}
Let $\A \in \mathsf{CM}$ and $c_1, \dots, c_n, d_1, \dots, d_n, a_1, \dots, a_{2n+1}, b \in A$ with $n \in \mathbb{N}$ be such that
\[
\A \vDash \psi_n(c_1, \dots, c_n, d_1, \dots, d_n, a_1, \dots, a_{2n+1}, b).
\]
Then the following conditions hold:
\benroman
\item\label{prop:zigzag facts:1} If $n>0$, then $\A \vDash \psi_{n-1}(c_2, \dots, c_n, d_2, \dots, d_n, d_1a_3, a_4, \dots, a_{2n+1}, b)$;
\item\label{prop:zigzag facts:2} $\A \vDash \psi_n(c_1, \dots, c_n, ed_1, \dots, ed_n, ea_1, a_2, \dots, a_{2n+1}, eb)$ for every $e \in A$.
\eroman
\end{Proposition}

\begin{proof}
   By the 
definition of $\psi_n$
the 
following holds in $\A$:
    \begin{align}
b &= a_1 c_1;\label{Isbell : 1 : formula x}\\
a_1 & = d_1 a_2;\label{Isbell : 2 : formula x}\\
a_{2i}c_i & = a_{2i+1}c_{i+1} \text{ for }i = 1, \dots, n-1;\label{Isbell : 3 : formula x}\\
d_i a_{2i+1} & = d_{i+1}a_{2(i+1)} \text{ for }i = 1, \dots, n-1;\label{Isbell : 4 : formula x}\\
a_{2n}c_n & = a_{2n+1};\label{Isbell : 5 : formula x}\\
d_n a_{2n+1}& = b.
\label{Isbell : 6 : formula x}
\end{align}

\eqref{prop:zigzag facts:1}: Assume that $n > 0$. We have two cases: either $n = 1$ or $n > 1$.
First, suppose that $n = 1$. Applying \eqref{Isbell : 6 : formula x} we obtain $b = d_1a_3$. So, the definition of $\psi_0$ yields $\A \vDash \psi_{0}(d_1a_3, b)$, as desired. Next, we consider the case where $n>1$.
\cref{Prop : Isb argument 0 -> y = 0} implies that $b=d_1a_3c_2$. From \eqref{Isbell : 4 : formula x} it follows that $d_1a_3=d_2a_4$. These two equalities, together with \eqref{Isbell : 3 : formula x}--\eqref{Isbell : 6 : formula x}, imply 
\[
\A \vDash \psi_{n-1}(c_2, \dots, c_n, d_2, \dots, d_n, d_1a_3, a_4, \dots, a_{2n+1}, b)
\]
by the definition of $\psi_{n-1}$.

\eqref{prop:zigzag facts:2}: Let $e \in A$. We have two cases: either $n = 0$ or $n > 0$. First, suppose that $n = 0$. Then the definition of $\psi_0$ yields $a_1=b$. It follows that $e a_1 = eb$. Hence, $\A \vDash \psi_0(ea_1, eb)$, as desired. Next, we consider the case where $n>0$.
 Multiplying \eqref{Isbell : 1 : formula x}, \eqref{Isbell : 2 : formula x}, \eqref{Isbell : 4 : formula x}, and \eqref{Isbell : 6 : formula x} by $e$ on both sides yields
\begin{align}
eb &= ea_1 c_1;\label{Isbell : 1e : formula x}\\
ea_1 & = ed_1 a_2;\label{Isbell : 2e : formula x}\\
ed_i a_{2i+1} & = ed_{i+1}a_{2(i+1)} \text{ for }i = 1, \dots, n-1;\label{Isbell : 4e : formula x}\\
ed_n a_{2n+1}& = eb.\label{Isbell : 6e : formula x}
\end{align}
Equations \eqref{Isbell : 1e : formula x}--\eqref{Isbell : 6e : formula x}, together with \eqref{Isbell : 3 : formula x} and \eqref{Isbell : 5 : formula x}, imply 
\[
\A \vDash \psi_n(c_1, \dots, c_n, ed_1, \dots, ed_n, ea_1, a_2, \dots, a_{2n+1}, eb)
\]
by the definition of $\psi_n$.
\end{proof}

\begin{Proposition}\label{Prop : Isbell arguments technical}
Let $\A \in \mathsf{CM}$ and $c_1, \dots, c_n, d_1, \dots, d_n, a_1, \dots, a_{2n+1}, b \in A$ with $n > 0$ be such that
\[
\A \vDash \psi_n(c_1, \dots, c_n, d_1, \dots, d_n, a_1, \dots, a_{2n+1}, b).
\]
Then the following conditions hold:
\benroman
\item\label{Prop : Isbell arguments technical:1} if $a_2 = 1$, then $\A \vDash \varphi_{n-1}(a_1a_3, a_4, \dots, a_{2n+1}, b)$; \label{item : a2 neq 1}
\item\label{Prop : Isbell arguments technical:2} if $d_1$ is invertible and $a_1 = a_2$, then $\A \vDash \varphi_{n-1}(a_3, \dots, a_{2n+1}, b)$. \label{item : a1 neq a2}
\eroman      
\end{Proposition}

\begin{proof}

\eqref{Prop : Isbell arguments technical:1}: Suppose that $a_2 = 1$. Since the definition of $\psi_{n}$ implies $a_1  = d_1 a_2$, we obtain  $a_1=d_1$. So,  from \cref{prop:zigzag facts}\eqref{prop:zigzag facts:1} it follows that 
\[
\A \vDash \psi_{n-1}(c_2, \dots, c_n, d_2, \dots, d_n, a_1a_3, a_4, \dots, a_{2n+1}, b).
\]
Thus, the definition of $\varphi_{n-1}$ ensures that $\A \vDash \varphi_{n-1}(a_1a_3, a_4, a_5, \dots, a_{2n+1}, b)$.

\eqref{Prop : Isbell arguments technical:2}: Suppose that $d_1$ is invertible and $a_1 = a_2$. From \cref{Prop : Isb argument 0 -> y = 0} and the assumption that $a_1 = a_2$ it follows that
\[
b = d_1a_2c_1 = d_1 a_1c_1.
\]
The definition of $\psi_{n}$ implies  $b = a_1 c_1 $, and so the above display yields $b = d_1 b$.
Therefore,
\[
\A \vDash \psi_n(c_1, \dots, c_n, d_1, \dots, d_n, a_1, \dots, a_{2n+1}, d_1b).
\]
From \cref{prop:zigzag facts}\eqref{prop:zigzag facts:1} and the above display it follows that
\[
\A \vDash \psi_{n-1}(c_2, \dots, c_n, d_2, \dots, d_n, d_1a_3, a_4, \dots, a_{2n+1}, d_1b).
\]
Applying \cref{prop:zigzag facts}\eqref{prop:zigzag facts:2} with $e=d_1^{-1}$ to the above display, we obtain 
\begin{equation}\label{Eq : Luca new display : last round}
    \A \vDash \psi_{n-1}(c_2, \dots, c_n, d_1^{-1}d_2, \dots, d_1^{-1}d_n, a_3, a_4, \dots, a_{2n+1}, b).
\end{equation}
Hence, the definition of $\varphi_{n-1}$ implies  $\A \vDash \varphi_{n-1}(a_3, \dots, a_{2n+1}, b)$.
\end{proof}

  We recall that, in the context of varieties, implicit operations admit a  description in terms of \emph{pp formulas}, that is, formulas of the form $\exists x_1, \dots, x_n \varphi$, where $\varphi$ is a conjunction of equations. More precisely,  if $f$ is an implicit operation of a variety $\K$, there exist implicit operations $f_1, \dots, f_n$ of $\K$ definable by pp formulas such that $f^\A = f_1^\A \cup \dots \cup f_n^\A$ for every $\A \in \K$ (see \cite[Cor.\ 3.10]{CKMIMP}). 
  Consequently, the pp definable implicit operations of $\K$ form the building blocks of all the implicit operations of $\K$. We denote the class of pp definable implicit operations of a class of algebras $\K$ by $\imppp(\K)$.

\begin{Proposition}\label{Prop : Isbell are implicit operations}
Let $\K$ be a variety of monoids and $n \in \mathbb{N}$. Then  $\varphi_n$ defines a $(2n+1)$-ary member of $\imppp(\K)$.
\end{Proposition}

\begin{proof}
For the case $\K = \mathsf{M}$, see \cite[Thm.\ 3.15]{CKMIMP}. As the restriction of an implicit operation of a class of algebras $\K$ to a subclass $\K' \subseteq\K$ is an implicit operation of $\K'$, we are done. 
\end{proof}

Let $\K$ be a class of algebras and $\Delta \subseteq \imppp(\K)$. We say that $\Delta$ is a
\emph{dominion base} for $\K$ when for all $\A \leq \B \in \K$ and $b \in \mathsf{d}_\K(\A, \B)$ there exist $f \in \Delta$ and $\langle a_1, \dots, a_n \rangle \in \mathsf{dom}(f^\B) \cap A^n$ such that $f^{\B}(a_1, \dots, a_n) = b$ (see \cite[Sec.\ 14]{CKMIMP}). 

\begin{exa}\label{Exa : dominion bases}
Given a variety $\K$ of commutative monoids and $n \in \mathbb{N}$, let $i_{\K, n}$ be the $(2n+1)$-ary member of $\imppp(\K)$ defined by $\varphi_n$ (see \cref{Prop : Isbell are implicit operations}). From \cref{Thm : Doms in CMon Vars} it follows that $\{ i_{\K, n} : n \in \mathbb{N} \}$ is a dominion base for $\K$.
\qed
\end{exa}

 Dominion bases and implicit operations are connected as follows (see \cite[Thm.\ 14.3]{CKMIMP}).

\begin{Theorem}\label{Thm : dominion bases}
    Let $\K$ be a variety with dominion base $\Delta$ and $f \in \imppp(\K)$ of arity $n$. Then there exist $g \in \Delta$ and $n$-ary terms $t_1, \dots, t_m$ of $\K$ such that for all $\A \in \K$ and $\langle a_1, \dots, a_n \rangle \in \mathsf{dom}(f^\A)$ we have
    \[
    \langle b_1, \dots, b_m\rangle\in \mathsf{dom}(g^\A)\, \, \text{ and }\, \, g^\A(b_1, \dots, b_m) = f^\A(a_1, \dots, a_n),
    \]
where $b_k = t_k^\A(a_1, \dots, a_n)$ for each $k \leq m$. 
\end{Theorem}

In view of \cref{Exa : dominion bases}, the following is a special instance of \cref{Thm : dominion bases}.

\begin{Corollary}\label{Cor : dominion base}
    Let $\K$ be a variety of commutative monoids and $f \in \imppp(\K)$ of arity $n$. Then there exist $m \in \mathbb{N}$ and $n$-ary monoid terms $t_1, \dots, t_{2m+1}$ such that for all $\A \in \K$ and $\langle a_1, \dots, a_n \rangle \in \mathsf{dom}(f^\A)$ we have
    \[
    \langle b_1, \dots, b_{2m+1}\rangle\in \mathsf{dom}(i_{\K, m}^\A)\, \, \text{ and }\, \, i_{\K, m}^\A(b_1, \dots, b_{2m+1}) = f^\A(a_1, \dots, a_n),
    \]
where $b_k = t_k^\A(a_1, \dots, a_n)$ for each $k \leq 2m+1$. 
\end{Corollary} 

\section{Extendable
implicit operations}

In general, the implicit operations of a variety $\K$ need not be total. Therefore, we say that an implicit operation $f$ of $\K$ is \emph{extendable} when for all $\A \in \K$ and  $\langle a_1, \dots, a_n \rangle \in \mathsf{dom}(f^\A)$ there exists an algebra $\B \in \K$ extending $\A$ such that $\langle a_1, \dots, a_n \rangle \in \mathsf{dom}(f^\B)$. The class of extendable implicit operations of $\K$ will be denoted by $\ext(\K)$, and that of pp definable extendable implicit operations of $\K$ by $\extpp(\K)$. The next result simplifies the task of defining extendable implicit operations (see 
\cite[Cor.\ 3.11 and Prop.\ 8.11(ii)]{CKMIMP}).

\begin{Proposition}\label{Prop : extendable trick} Let $\K$ be a variety and $\varphi(x_1, \dots, x_n, y)$ a pp formula such that for all $\A \in \K_\textsc{si}$ and $a_1, \dots, a_n \in A$ there exists a unique $b \in A$ such that $\A \vDash \varphi(a_1, \dots, a_n, b)$. Then $\varphi$ defines an $n$-ary member of $\extpp(\K)$.
\end{Proposition}

Notably, every member of a variety $\K$ can be 
``upgraded'' to one in which all the extendable implicit operations are total, in the following sense (see   \cite[Prop.\ 8.1 and Thm.\ 8.4]{CKMIMP}).

\begin{Theorem}\label{Thm : extendable : universal : SI}
    Let $\K$ be a variety and $\A \in \K$. Then there exists $\B \in \K$  with $\A \leq \B$ such that $f^\B$ is total and extends $f^\A$ for each $f  \in \ext(\K)$. In addition, when $\A \in \K_\textsc{si}$, the algebra $\B$ can be chosen in $\K_\textsc{si}$.
\end{Theorem}

In order to add a family of implicit operations $\mathcal{F} \subseteq \extpp(\K)$ to a variety $\K$, we proceed as follows. Let $\L$ be the language of $\K$ and $\L_\F$ the language obtained by adding to $\L$ a new $n$-ary function symbol $g_f$ for each $n$-ary $f \in \F$. Then, we expand every member $\A$ of $\K$ in which $\{ f^\A : f \in \F \}$ is a family of total functions to an algebra $\A[\L_\F]$ in the language $\L_\F$ by interpreting  $g_f$ as $f^\A$ for each $f \in \F$. In addition, for every $\mathsf{N} \subseteq \K$  let
\[
\mathsf{N}[\L_\F] = \{ \A[\L_\F] : \A \in \mathsf{N} \text{ and }\{ f^\A : f \in \F \} \text{ is a family of total functions} \}.
\]
The \emph{pp expansion} of $\K$ induced by $\F$ is $\SSS(\K[\L_\F])$ and represents the result of ``adding the implicit operations in $\mathcal{F}$ to $\K$''. 

\begin{exa}\label{Exa : inverses are interpolable}
Let $\K$ be a proper subvariety of $\mathsf{CM}$. We will prove that there exists a unary $f \in \extpp(\K)$ such that for every $\A \in \K$ and $a \in A$,
\[
\text{if $a^{-1}$ exists, then }a \in \mathsf{dom}(f^\A) \text{ and }f^\A(a) = a^{-1}.
\]
Consequently, $\mathcal{F} = \{ f \}$ induces a pp expansion of $\mathsf{K}$ in which the implicit operation of “taking  inverses” is interpolated by the term $g_f(x)$.

To prove the claim, recall that $\K$ is a proper subvariety of $\mathsf{CM}$. Together with
\cref{Prop : variety equalities} and \cref{Thm : classification of CMon-Vars}, this ensures the existence of  $m, n \in \mathbb{N}$ with $m > 0$ such that $\K = \V_{m, n}$. Then consider the pp formula
\[
\varphi(x,y) = (x^{n+1}y \thickapprox x^n) \sqcap (y^2x \thickapprox y).
\]
We will prove that $\varphi$ defines a unary member $f$ of $\extpp(\K)$. In view of \cref{Prop : extendable trick}, it suffices to show that for all $\A \in \K_\textsc{si}$ and $a \in A$ there exists a unique $b \in A$ such that $\A \vDash \varphi(a, b)$. So, consider $\A \in \K_\textsc{si}$ and $a \in A$. By \cref{Thm : char of SIs: inv or nil} we have two cases: either $a$ is $n$-nilpotent or it is invertible. 

First, suppose that $a$ is $n$-nilpotent, i.e., that $a^n$ is a zero element for $\A$. Together with the definition of $\varphi$, this yields $\A \vDash \varphi(a, a^n)$. Then consider $b \in A$ such that $\A \vDash \varphi(a, b)$. We need to prove that $b = a^n$. By the definition of $\varphi$ we have $b^2 a = b$. 
We will show by induction on $k$ that $b=b^{2^k}a^{2^k-1}$ for every $k \in \mathbb{N}$. In the base case, we have $k=0$, and the equality is straightforward to verify. For the inductive step, suppose that $b=b^{2^k}a^{2^k-1}$ for some $k \in \mathbb{N}$. We need to show that $b=b^{2^{k+1}} a^{2^{k+1}-1}$. Using the induction hypothesis and the identity $b^2 a = b$, we obtain
\[
b = b^{2^k}a^{2^k-1} = (b^2 a)^{2^k}a^{2^k-1} = b^{2^{k+1}} a^{2^k}a^{2^k-1} =b^{2^{k+1}} a^{2^{k+1}-1},
\]
as desired. Then consider $k \in \mathbb{N}$ such that $n \leq 2^{k+1}-1$. We have $b=c a^n$, where $c= b^{2^{k+1}} a^{2^{k+1}-n-1}$. Hence, $b=a^n$ because $a^n$ is a zero element of $\A$.

Next, we consider the case where $a$ is invertible. Together with the definition of $\varphi$, this yields $\A \vDash \varphi(a, a^{-1})$. Then consider $b \in A$ such that $\A \vDash \varphi(a, b)$. We need to prove that $b = a^{-1}$. By the definition of $\varphi$, we have $a^{n+1}b = a^n$. This yields
\[
b = (a^{-1})^{n+1}a^{n+1}b = (a^{-1})^{n+1}a^n = a^{-1}.
\]
Hence, we conclude that $\varphi$ defines a unary member $f$ of $\extpp(\K)$, as desired.

Lastly, from the definition of $\varphi$ it follows that $\A \vDash \varphi(a, a^{-1})$ for all $\A \in \K$ and $a\in A$ for which $a^{-1}$ exists. As $\varphi$ defines $f$, in this case we  have $a \in \mathsf{dom}(f^\A)$ and $f^{\A}(a) = a^{-1}$.
\qed
\end{exa}

\section{Beth companions}

A pp expansion of a variety $\K$ is said to be a \emph{Beth companion} of $\K$ when it has the strong Beth definability property or, equivalently, the strong epimorphism surjectivity property (see \cite[Thm.\ 11.6]{CKMIMP}).\ While a variety need not have a Beth companion, up to term equivalence it may possess only one  (see \cite[Thm.\ 11.7]{CKMIMP}). For this reason, we talk about \emph{the} Beth companion of $\K$ (when it exists). Our aim is to prove the following.

\begin{Theorem} \label{Thm : main CMon Vars}
    The following conditions are equivalent for a variety $\mathsf{K}$ of commutative monoids: 
    \benroman
    \item\label{item : main : 1} $\mathsf{K}$ has a Beth companion; 
    \item\label{item : main : 2} $\mathsf{K}$ has the strong epimorphism surjectivity property; 
    \item\label{item : main : 3} $\mathsf{K}$ is inverse.
    \eroman
\end{Theorem}

\begin{Remark}
    As a variety is its own Beth companion precisely when it has the strong epimorphism surjectivity property (see \cite[Thms.\ 11.9(vi) and 11.6]{CKMIMP}),
we deduce that if a variety of commutative monoids $\K$ has a Beth companion, then $\K$ is its own Beth companion.
\qed
\end{Remark}

The proof of \cref{Thm : main CMon Vars} relies on the next observation (see \cite[Prop.\ 14.6]{CKMIMP}).

\begin{Proposition}\label{Prop : strange condition : extend}
Let $\K$ be a variety with a Beth companion of the form $\SSS(\K[\L_\F])$. Then for every $f \in \imppp(\K)$ there exists $g \in \extpp(\K)$ such that for all $\A[\L_\F] \in \K[\L_\F]$ and $\langle a_1, \dots, a_n \rangle \in \mathsf{dom}(f^\A)$ we have
\[
\langle a_1, \dots, a_n \rangle \in \mathsf{dom}(g^\A) \, \, \text{ and }\, \, f^\A(a_1, \dots, a_n) = g^\A(a_1, \dots, a_n).
\]
\end{Proposition}

We are now ready to prove \cref{Thm : main CMon Vars}

\begin{proof}   
    (\ref{item : main : 3})$\Rightarrow$(\ref{item : main : 2}): It suffices to show that $\mathsf{d}_\K(\A, \B) = A$ for all $\A \leq \B \in \mathsf{K}$. So, consider $\A \leq \B \in \K$. Observe that $\A_\mathsf{s} \leq \B_\mathsf{s} \in \mathsf{S}$. Moreover, $\A_\mathsf{s}$ is an inverse semigroup because $\A$ is an inverse monoid by assumption. Hence, $\mathsf{d}_\mathsf{S}(\A_\mathsf{s}, \B_\mathsf{s}) = A$ by \cref{Thm : inverse semigroups : absolutely closed}. Together with Theorems \ref{Thm : zigzag original} and \ref{Thm : Doms in CMon Vars}, this yields $\mathsf{d}_\K(\A, \B) = A$.

(\ref{item : main : 2})$\Rightarrow$(\ref{item : main : 1}): Since $\K$ is a pp expansion of itself, the assumption ensures that $\K$ is its own Beth companion.

    (\ref{item : main : 1})$\Rightarrow$(\ref{item : main : 3}): Let $\B$ be the commutative monoid\footnote{The definition of the monoid $\B$ is reminiscent of Higgings' \cite[Sec.\ 2]{HigcSG}.} with universe
    \[
    B = \{ 0, 1 \}^3 \cup  \{0\},
    \]
    neutral element $\langle 0, 0, 0 \rangle$, and multiplication defined for every $a, b \in B$ as
\[
ab = \begin{cases}
\langle k_1 + k_2, m_1 + m_2, n_1 + n_2 \rangle &\text{ if }a = \langle k_1, m_1, n_1 \rangle, b= \langle k_2, m_2, n_2 \rangle,\\
& \text{ and }\langle k_1 + k_2, m_1 + m_2, n_1 + n_2\rangle \in \{ 0, 1 \}^3;\\
0 &\text{ otherwise.}
\end{cases}
\]

We begin with the following observation.
\begin{Claim}\label{Claim : B invertible}
  The monoid  $\B$ is a subdirectly irreducible member of $\mathsf{C}_2$.
\end{Claim}

\begin{proof}[Proof of the Claim]
    The definition of $\B$ ensures that it is a commutative monoid and $\B \vDash x^2 \thickapprox x^3$. As $\mathsf{C}_2$ is axiomatized relative to $\mathsf{CM}$ by $x^2 \thickapprox x^3$, we obtain $\B \in \mathsf{C}_2$. Next, we will show that $\B$ is subdirectly irreducible. To this end, it will be enough to prove that
    \[
    \langle 0, \langle 1, 1, 1 \rangle \rangle \in \bigcap (\mathsf{Con}(\B) - \{\mathsf{id}_{B}\}).
    \]
    Consider $\theta \in \mathsf{Con}(\B) - \{\mathsf{id}_{B}\}$. We will prove that $\langle 0, \langle 1, 1, 1 \rangle \rangle \in \theta$. As $\theta \ne \mathsf{id}_B$, there exists $\langle a, b \rangle \in \theta$ with $a \ne b$. We have two cases: either $0 \in \{ a, b \}$ or $0 \notin \{ a, b \}$. Suppose first that $0 \in \{ a, b \}$. By symmetry we may assume that $a = 0$. Since $b \ne a = 0$ and $B = \{ 0, 1 \}^3 \cup \{ 0 \}$, we have $b = \langle k, m, n \rangle$ for some $k, m, n \in \{ 0, 1 \}$. Then $\langle 1-k, 1-m, 1-n \rangle \in \{ 0, 1 \}^3 \subseteq B$. From $\langle 0, \langle k, m, n \rangle\rangle = \langle a, b \rangle \in \theta$ it follows that
    \[
    \langle 0, \langle 1, 1, 1 \rangle \rangle = \langle 0 \cdot \langle 1-k, 1-m, 1-n \rangle, \langle k, m, n \rangle \cdot \langle 1-k, 1-m, 1-n \rangle\rangle \in \theta.
    \]
Lastly, we consider the case where $0 \notin \{ a, b \}$. As $B = \{ 0, 1 \}^3 \cup \{ 0 \}$, we have  $a = \langle k_1, m_1, n_1 \rangle$ and $b = \langle k_2, m_2, n_2 \rangle$ with $k_i, m_i, n_i \in \{ 0, 1 \}$. 
Since $a \ne b$, by symmetry we may assume that $k_1 \ne k_2$ and, in particular, that $k_1 < k_2$. The latter means that $k_1 = 0$ and $k_2 = 1$. Then
consider $\langle 1-k_1, 1-m_1, 1-n_1 \rangle \in \{ 0, 1 \}^3 \subseteq B$. Since $k_2 = 1$ and $1-k_1 = 1-0 = 1$, we have $k_2 + (1 - k_1) = 2$, whence $\langle k_2, m_2, n_2 \rangle \cdot \langle 1-k_1, 1-m_1, 1-n_1 \rangle = 0$. On the other hand, $\langle k_1, m_1, n_1 \rangle \cdot \langle 1-k_1, 1-m_1, 1-n_1 \rangle = \langle 1, 1, 1 \rangle$. Together with $\langle \langle k_1, m_1, n_1 \rangle, \langle k_2, m_2, n_2 \rangle\rangle = \langle a, b \rangle \in \theta$, this yields $\langle 0, \langle 1, 1, 1 \rangle \rangle \in \theta$.
\end{proof}

Suppose, with a view to contradiction,  that $\K$ has a Beth companion and is not inverse.
Let $\SSS(\K[\mathscr{L}_\mathcal{F}])$ be the Beth companion of $\K$ and recall that $\mathcal{F} \subseteq \extpp(\K)$ by the definition of a pp expansion. Moreover, recall that $i_{\K, 1}$ is a ternary member of $\imppp(\K)$ defined by the Isbell's formula $\varphi_1$ (see \cref{Prop : Isbell are implicit operations} and \cref{Exa : dominion bases}). By \cref{Prop : strange condition : extend} there exists $g \in \extpp(\K)$ such that 
    \begin{equation} \label{Eq : g interpol f_1}
        \langle a_1, a_2, a_3 \rangle \in \mathsf{dom}(g^{\C}) \, \, \text{ and } \, \, i_{\K, 1}^\C(a_1, a_2, a_3) = g^{\C}(a_1, a_2, a_3)
    \end{equation}
    for all $\langle a_1, a_2, a_3 \rangle \in \mathsf{dom}(i_{\K, 1}^\C)$ and $\C[\L_\F] \in \K[\L_\F]$.

    As $\K$ is not inverse, 
    \cref{Thm : char of inverse monoids} and \cref{Claim : B invertible} guarantee that $\B \in \K_\textsc{si}$. Therefore, we can apply \cref{Thm : extendable : universal : SI}, obtaining some 
   \begin{equation} \label{Eq : properties of D}
    \A \in \K_\textsc{si} \text{ with }\B \leq \A \text{ such that }f^\A \text{ is total for every }f \in \mathcal{F} \cup \{ g \}.
\end{equation} 

\begin{Claim}\label{Claim : zero element for A}
The monoid $\A$ possesses a zero element, namely, $0$.
\end{Claim}

\begin{proof}[Proof of the Claim]
Since $\B \leq \A$, we have $0=0^2$ and $0 \neq 1$. As $\A$ is subdirectly irreducible, \cref{Prop : zeros preserved in extensions} implies that $0$ is a zero element of $\A$.
\end{proof}

\begin{Claim}\label{Claim : value of 1st Isbell}
    We have 
    \[
    \langle \langle 1, 1, 0 \rangle, \langle 0, 1, 0 \rangle, \langle 0, 1, 1 \rangle \rangle \in \mathsf{dom}(i_{\K, 1}^\A) \, \, \text{ and }\, \, i_{\K, 1}^\A(\langle 1, 1, 0 \rangle, \langle 0, 1, 0 \rangle, \langle 0, 1, 1 \rangle) = \langle 1, 1, 1\rangle.
\]
\end{Claim}

\begin{proof}[Proof of the Claim]
From the definitions of $\B$ and $\psi_1$ 
it follows that
\[
\B \vDash \psi_1(\langle 0, 0, 1\rangle, \langle 1, 0, 0 \rangle, \langle 1, 1, 0 \rangle, \langle 0, 1, 0 \rangle, \langle 0, 1, 1 \rangle, \langle 1, 1, 1 \rangle).
\]
By the definition of $\varphi_1$ this amounts to
\[
\B  \vDash \varphi_1(\langle 1, 1, 0 \rangle, \langle 0, 1, 0 \rangle, \langle 0, 1, 1 \rangle, \langle 1, 1, 1 \rangle).
\]
As the pp formula $\varphi_1$ defines the implicit operation $i_{\K, 1}$  and $\B \leq \A$, the  statement follows because the validity of pp formulas is preserved by extensions.
\end{proof}

Moreover, recall that $g \in \imppp(\K)$ is ternary and observe that $g^\A$ is total by (\ref{Eq : properties of D}). Therefore, by \cref{Cor : dominion base} there exist $k \in \mathbb{N}$ and ternary monoid terms $t_1, \dots, t_{2k+1}$ such that
\[
        g^\A(a_1, a_2, a_3) = i_{\K, k}^\A(t_1^\A(a_1, a_2, a_3), \dots, t_{2k+1}^\A(a_1, a_2, a_3))\text{ for all } a_1, a_2, a_3 \in A.
\]
Consequently, there exists also
\begin{equation} \label{Eq : min assumption}
 \begin{split}
     n = \min \{ & m \in \mathbb{N} : \text{there exist ternary monoid terms }t_1, \dots, t_{2m+1} \text{ such that}\\
     &g^\A(a_1, a_2, a_3) = i_{\K, m}^\A(t_1^\A(a_1, a_2, a_3), \dots, t_{2m+1}^\A(a_1, a_2, a_3)) \\
     &\text{for all }a_1, a_2, a_3 \in A  - \{0\}\}.
 \end{split}   
\end{equation}

\begin{Claim} \label{Claim : terms in f_m are variables}
We may assume that $t_m \in \{ x_1, x_2, x_3, 1 \}$ for every $m \leq 2n+1$.
    \end{Claim}
    \begin{proof} [Proof of the Claim]
Consider $m \leq 2n+1$. As $\A$ is a commutative monoid, we may assume that $t_m$ is of the form $x_1^{k_{1}} x_2^{k_{2}}x_3^{k_{3}}$ for some $k_{1}, k_{2}, k_{3} \in \mathbb{N}$. If $k_1 + k_2 + k_3 \leq 1$, then we may assume that $t_m \in \{ x_1, x_2, x_3, 1 \}$, as desired. Then suppose that $k_1 + k_2 + k_3 > 1$, with a view to contradiction. Recall that $\{ 0, 1 \}^3 \subseteq B \subseteq A$. 
By applying in succession the fact that $\B \leq \A$ (which holds by (\ref{Eq : properties of D})), the assumption that $t_m = x_1^{k_1}x_2^{k_2}x_3^{k_3}$, and the definition of $\B$ in combination with $k_1 + k_2 + k_3 > 1$, we obtain
\begin{align*}
    t_m^\A(\langle 1, 1, 0 \rangle, \langle 0, 1, 0 \rangle, \langle 0, 1, 1 \rangle) &= t_m^\B(\langle 1, 1, 0 \rangle, \langle 0, 1, 0 \rangle, \langle 0, 1, 1 \rangle) \\
    &=  \langle 1, 1, 0 \rangle^{k_1}\langle 0, 1, 0 \rangle^{k_2} \langle 0, 1, 1 \rangle^{k_3}\\
    &= 0.
\end{align*}
Together with \cref{Claim : zero element for A}, the above display yields that 
\begin{equation}\label{Eq : t is a zero element}
    0 = t_m^\A(\langle 1, 1, 0 \rangle, \langle 0, 1, 0 \rangle, \langle 0, 1, 1 \rangle) \text{ is a zero element for }\A.
\end{equation}

Now, observe that $\A \in \K$ and that the algebra $\A[\L_\F]$ is defined by  \eqref{Eq : properties of D}. Therefore, $\A[\L_\F] \in \K[\L_\F]$.  Together with \eqref{Eq : g interpol f_1} and Claim \ref{Claim : value of 1st Isbell}, this yields
\[
\langle \langle 1, 1, 0 \rangle, \langle 0, 1, 0 \rangle, \langle 0, 1, 1 \rangle \rangle \in \mathsf{dom}(g^\A) \, \, \text{ and }\, \, g^\A(\langle 1, 1, 0 \rangle, \langle 0, 1, 0 \rangle, \langle 0, 1, 1 \rangle) = \langle 1, 1, 1\rangle.
\]
By \eqref{Eq : min assumption} this implies
\[
i_{\K, n}^\A(t_1^\A(\langle 1, 1, 0 \rangle, \langle 0, 1, 0 \rangle, \langle 0, 1, 1 \rangle), \dots, t_{2n+1}^\A(\langle 1, 1, 0 \rangle, \langle 0, 1, 0 \rangle, \langle 0, 1, 1 \rangle)) = \langle 1, 1, 1\rangle.
\]
Since $i_{\K, n}$ is defined by $\varphi_n$, we obtain
\[
\A \vDash \varphi_n(t_1^\A(\langle 1, 1, 0 \rangle, \langle 0, 1, 0 \rangle, \langle 0, 1, 1 \rangle), \dots, t_{2n+1}^\A(\langle 1, 1, 0 \rangle, \langle 0, 1, 0 \rangle, \langle 0, 1, 1 \rangle), \langle 1, 1, 1\rangle).
\]
Lastly, applying \cref{Cor: : arg in Isb = 0 - > value 0} and \eqref{Eq : t is a zero element} to the above display, we conclude that $0 = \langle 1, 1, 1 \rangle$, which is false.
    \end{proof}

\begin{Claim} \label{Claim : excluded cases}
We have $n \ne 0$,   $t_2 \neq 1$, and $t_1 \neq t_2$.
\end{Claim}
\begin{proof}[Proof of the Claim]
    We begin by showing that $n \neq 0$. Suppose the contrary, with a view to contradiction. On the one hand, Claim \ref{Claim : terms in f_m are variables} implies 
    \begin{equation}\label{Eq : new equation : last steps : claim 6.8} 
    t_1^\A(\langle 1,1,0 \rangle, \langle 0,1,0 \rangle, \langle 0,1,1 \rangle) \in \{\langle 0,0,0 \rangle,\langle 1,1,0 \rangle, \langle 0,1,0 \rangle, \langle 0,1,1 \rangle\}.
        \end{equation}
On the other hand, applying in succession \eqref{Eq : g interpol f_1} and \cref{Claim : value of 1st Isbell}, we obtain

    \begin{align*}
g^\A(\langle 1,1,0 \rangle, \langle 0,1,0 \rangle, \langle 0,1,1 \rangle) &= i_{\K, 1}^\A(\langle 1,1,0 \rangle, \langle 0,1,0 \rangle, \langle 0,1,1 \rangle)\\
&= \langle 1, 1, 1\rangle.
    \end{align*}
Together with \eqref{Eq : min assumption}, the assumption that $n = 0$, and the definition of $i_{\K, 0}$, the above display yields 
    \begin{align*}
        t_1^\A(\langle 1,1,0 \rangle, \langle 0,1,0 \rangle, \langle 0,1,1 \rangle) &= i_{\K,0}^{\A}(t_1^\A(\langle 1,1,0 \rangle, \langle 0,1,0 \rangle, \langle 0,1,1 \rangle)) \\
        &= g^\A(\langle 1,1,0 \rangle, \langle 0,1,0 \rangle, \langle 0,1,1 \rangle)\\
        & =\langle 1, 1, 1\rangle,
            \end{align*}
a contradiction with \eqref{Eq : new equation : last steps : claim 6.8}. Hence, we conclude that $n \geq 1$, as desired.

Next, we prove that $t_2 \neq 1$ and $t_1 \neq t_2$. Suppose, with a view to contradiction, that either $t_2 = 1$ or $t_1 = t_2$. We will reach a contradiction with \eqref{Eq : min assumption} and the fact that the Isbell's formula $\varphi_k$ defines $i_{\K, k}$ for each $k \in \mathbb{N}$ by  showing that there exist ternary monoid terms $s_1, \dots, s_{2n-1}$ such that 
        \begin{equation} \label{Eq : zigzag with s terms}
            \A \vDash \varphi_{n-1}(s_1^{\A}(a_1, a_2, a_3), \dots, s_{2n-1}^{\A}(a_1, a_2, a_3), g^{\A}(a_1, a_2, a_3))
        \end{equation}
        for all $a_1, a_2, a_3 \in A - \{0\}$.

We have two cases depending on whether $t_2 = 1$ or $t_1 = t_2$. We begin with the case where $t_2 = 1$. By \eqref{Eq : min assumption} and the fact that $\varphi_n$ defines $i_{\K, n}$ we obtain that for all $a_1, a_2, a_3 \in A - \{ 0 \}$,
\[
\A \vDash \varphi_n(t_1^{\A}(a_1, a_2, a_3), 1, t_3^{\A}(a_1, a_2, a_3), \dots, t_{2n+1}^{\A}(a_1, a_2, a_3), g^\A(a_1, a_2, a_3)).
\]
 Therefore, we can apply \cref{Prop : Isbell arguments technical} (in particular, its case \eqref{item : a2 neq 1}), obtaining that \eqref{Eq : zigzag with s terms} holds for all $a_1, a_2, a_3 \in A - \{ 0 \}$ letting $s_1 = t_1t_3$ and $s_i = t_{i+2}$ for $1 < i \leq 2n-1$.       

        Next, we consider the case where $t_1 = t_2$. We begin with the following observation.
 
        \begin{Subclaim}\label{Subclaim : only one}
Let $a_1, a_2, a_3 \in A - \{ 0 \}$. Then there exist $c_1, \dots, c_n, d_1, \dots, d_{n} \in A$ with $d_1$ invertible such that $t_1^\A(a_1, a_2, a_3) = t_2^\A(a_1, a_2, a_3)$ and 
\[
\A \vDash \psi_n(c_1, \dots, c_n, d_1, \dots, d_n, t_1^\A(a_1, a_2, a_3), \dots, t_{2n+1}^\A(a_1, a_2, a_3), g^\A(a_1, a_2, a_3)).
\]
        \end{Subclaim}

\begin{proof}[Proof of the Subclaim]
   From \eqref{Eq : min assumption}, the fact that $\varphi_n$ defines $i_{\K, n}$, and the definition of $\varphi_n$ it follows that there  exist $c_1, \dots, c_n, d_1, \dots, d_{n} \in A$ such that
   \[
   \A \vDash \psi_n(c_1, \dots, c_n, d_1, \dots, d_n, t_1^\A(a_1, a_2, a_3), \dots, t_{2n+1}^\A(a_1, a_2, a_3), g^\A(a_1, a_2, a_3)).
   \]
   Furthermore, the assumption that $t_1 = t_2$ ensures that  $t_1^\A(a_1, a_2, a_3) = t_2^\A(a_1, a_2, a_3)$. Therefore, it only remains to show that $d_1$ is invertible. To this end, observe that equation \eqref{Isbell : 2 : formula} in $\psi_n$ and  $t_1^{\A}(a_1, a_2, a_3) = t_2^{\A}(a_1, a_2, a_3)$ imply  $t_1^{\A}(a_1, a_2, a_3) = d_1t_2^{\A}(a_1, a_2, a_3) = d_1t_1^{\A}(a_1, a_2, a_3)$. 
        Consequently,
\[
            t_1^{\A}(a_1, a_2, a_3) = d_1^{\?m}t_1^{\A}(a_1, a_2, a_3) \text{ for every } m \in \mathbb{N}.
\]
Recall from \eqref{Eq : properties of D} that $\A$ is subdirectly irreducible. Therefore, Theorem \ref{Thm : char of SIs: inv or nil} yields that $d_1$ is either nilpotent or invertible. 
        Suppose, with a view to contradiction, that
        $d_1$ is $m$-nilpotent for some $m \in \mathbb{N}$. Together with the above display, this yields  $t_1^{\A}(a_1, a_2, a_3) = d_1^mt_1^{\A}(a_1, a_2, a_3) = d_1^m$. As $d_1$ is $m$-nilpotent, we conclude that $t_1^{\A}(a_1, a_2, a_3)$ is a zero element for $\A$. Together with 
        \cref{Prop : basic properties}\eqref{item : basic properties : 1} and Claim \ref{Claim : zero element for A}, this implies $t_1^{\A}(a_1, a_2, a_3) = 0$. 
        Moreover, recall that $t_1 \in \{x_1, x_2, x_3, 1\}$ by Claim \ref{Claim : terms in f_m are variables}. 
        Hence, $t_1^\A(a_1, a_2, a_3) \in \{ a_1, a_2, a_3, 1 \}$. Since $a_1, a_2, a_3 \in A - \{ 0 \}$ by assumption and $t_1^\A(a_1, a_2, a_3) = 0$, we deduce $0 = t_1^\A(a_1, a_2, a_3) = 1$. Therefore, $\A$ is trivial, a contradiction with the assumption that it is subdirectly irreducible.
        Thus, we conclude that $d_1$ is invertible.
\end{proof}

In view of \cref{Subclaim : only one}, 
we can apply \cref{Prop : Isbell arguments technical} (in particular, its case \eqref{item : a1 neq a2}), obtaining that  \eqref{Eq : zigzag with s terms} holds for all $a_1, a_2, a_3 \in A - \{ 0 \}$ letting $s_1 = t_3$ and $s_i = t_{i+2}$ for $1 < i \leq 2n-1$.  
\end{proof}

    In view of Claim \ref{Claim : excluded cases},  we have $t_1 \neq t_2$ and $t_2 \neq 1$. 
    Therefore, Claim \ref{Claim : terms in f_m are variables} implies $t_2 \in \{x_1, x_2, x_3\}$. Moreover, by symmetry we may assume $t_2 = x_1$.\footnote{The cases where $t_2$ is $x_2$ or $x_3$ are handled analogously by changing the order of the arguments $\langle 1, 0, 0 \rangle$, $\langle 0, 0, 0 \rangle$, and $\langle 0, 0, 0 \rangle$ in the next proof in the natural way.}
    By (\ref{Eq : min assumption}) we know that 
\begin{align*}
i_{\K,n}^{\A}(t_1^\A(&\langle 1,0,0 \rangle,\langle 0, 0, 0 \rangle,\langle 0, 0, 0 \rangle), \dots, t_{2n+1}^\A(\langle 1,0,0 \rangle,\langle 0, 0, 0 \rangle,\langle 0, 0, 0 \rangle))\\
&= g^\A(\langle 1,0,0 \rangle,\langle 0, 0, 0 \rangle,\langle 0, 0, 0 \rangle).
\end{align*}
Recall that $i_{\K, n}$ is defined by $\varphi_n$. Together with the above display and the definition of $\varphi_n$, this yields the existence of some $c_1, \dots, c_n, d_1, \dots, d_n \in A$ such that
\begin{align*}
    \A \vDash \psi_n(&c_1, \dots, c_n, d_1, \dots, d_n, t_1^\A(\langle 1,0,0 \rangle,\langle 0, 0, 0 \rangle,\langle 0, 0, 0 \rangle), \dots\\
    & \dots, t_{2n+1}^\A(\langle 1,0,0 \rangle,\langle 0, 0, 0 \rangle,\langle 0, 0, 0 \rangle), g^\A(\langle 1,0,0 \rangle,\langle 0, 0, 0 \rangle,\langle 0, 0, 0 \rangle)).
\end{align*}
Consequently, the equation \eqref{Isbell : 2 : formula} in $\psi_n$ ensures that 
\[
t_1^\A(\langle 1,0,0 \rangle,\langle 0, 0, 0 \rangle,\langle 0, 0, 0 \rangle) = d_1 t_2^\A(\langle 1,0,0 \rangle,\langle 0, 0, 0 \rangle,\langle 0, 0, 0 \rangle).
\]
    Moreover, \cref{Claim : terms in f_m are variables} guarantees that $t_1 \in \{x_2, x_3, 1\}$ because  $t_1 \neq t_2 = x_1$ by assumption. Applying in succession the fact that $\langle 0, 0, 0 \rangle$ is the neutral element of $\A$, $t_1 \in \{ x_2, x_3, 1 \}$, the above display, and $t_2 = x_1$, we obtain
\[
    1^{\A} = \langle 0,0,0 \rangle = t_1^\A(\langle 1,0,0 \rangle,\langle 0, 0, 0 \rangle,\langle 0, 0, 0 \rangle) = d_1 t_2^\A(\langle 1,0,0 \rangle,\langle 0, 0, 0 \rangle,\langle 0, 0, 0 \rangle) = d_1\langle 1,0,0 \rangle.
\]
Lastly, we have
    \[
    1^\A = (1^\A)^2 = d^2\langle 1,0,0 \rangle^2 = d^2 \cdot 0 = 0,
    \]
where the first equality  holds because $1^\A$ is the neutral element of $\A$, 
the second holds by the previous display, the third by the definition of $\B$, and the fourth by \cref{Claim : zero element for A}. Observe that the above display implies that $\A$ is trivial, a contradiction with the fact that $\langle 0, 0, 0 \rangle \ne 0$ are both elements of $\A$. 
\end{proof}

\begin{Remark}
Let $\K$ be a variety. We recall that a homomorphism $f \colon \A \to \B$ with $\A, \B \in \K$ is a $\K$-\emph{epimorphism} when for every pair of homomorphisms $g, h \colon \B \to \C$ with $\C \in \K$,
\[
g\circ f = h \circ f \text{ implies }g = h
\]
(see, e.g.,  \cite[Def.~7.39]{AHS06}).\ We say that $\K$ has
\benroman
\item the \emph{epimorphism surjectivity property} when all $\K$-epimorphisms are surjective;
\item the \emph{weak epimorphism surjectivity property} when all $\K$-epimorphisms between finitely generated algebras are surjective (see, e.g., \cite[p.\ 259]{HMTII} and \cite{CKMWES}).
\eroman
Equivalently, $\K$ has the (resp.\ weak) epimorphism surjectivity property when for all (resp.\ finitely generated) $\A \leq \B \in \K$ we have either $A = B$ or $\mathsf{d}_\K(\A, \B) \ne B$. As their names suggest, the strong epimorphism surjectivity property is stronger than the epimorphism surjectivity property which, in turn, is stronger than the weak epimorphism surjectivity property.
However, the three properties are equivalent when $\K$ has the amalgamation property (see, e.g.,  \cite[Thm.~1.3]{BMRES} and \cite[Thm.~7.14]{CKMIMP}).

In view of \cref{Thm : main CMon Vars}, a variety of commutative monoids has the strong epimorphism surjectivity property precisely when it is inverse. It is therefore natural to wonder whether there exist noninverse varieties of commutative monoids with the (resp.\ weak) epimorphism surjectivity property. In this remark, we provide an exhaustive answer to this question.

On the one hand, a variety of commutative monoids has the epimorphism surjectivity property precisely when it is inverse, as we proceed to illustrate. Let $\K$ be a noninverse variety of commutative monoids. By \cref{Thm : char of inverse monoids} we have $\mathsf{C}_2 \subseteq  \K$. In \cite[Exa.\ 2]{HigcSG}, Higgins exhibited a pair of a commutative semigroups $\A$ and $\B$ satisfying $x^2 \approx x^3$ such that $\mathsf{d}_{\mathsf{CS}}(\A, \B) = B \ne A$.
Adding  a new element $1$ that acts as a  
neutral element for both $\A$ and $\B$, we can turn them into a pair of monoids $\A^1 \leq  \B^1 \in \mathsf{C}_2 \subseteq \K$ with $A \ne B$. 
 As $\A \leq \A^1_{\mathsf{s}}$ and $\B \leq \B^1_{\mathsf{s}}$, we have 
\begin{equation} \label{Eq : dom inclusion}
B = \mathsf{d}_{\mathsf{CS}}(\A, \B) \subseteq \mathsf{d}_{\mathsf{CS}}(\A^1_{\mathsf{s}}, \B^1_{\mathsf{s}}) = \mathsf{d}_{\mathsf{CM}}(\A^1, \B^1),
\end{equation}
where the first step holds by assumption, the second by \cite[Cor.\ 4.6(i)]{CKMIMP}, and the third by \cref{Prop : dom computed in CSG vs CMon}.\ 
Moreover, observe that $1 \in A^1 \subseteq \mathsf{d}_{\mathsf{CM}}(\A^1, \B^1)$. Together with \eqref{Eq : dom inclusion}, this yields $\mathsf{d}_{\mathsf{CM}}(\A^1, \B^1) = B^1 \neq A^1$. 
Hence, $\K$ lacks the epimorphism surjectivity property, as desired (the inverse varieties of commutative monoids have this property by \cref{Thm : main CMon Vars}).

On the other hand, all proper subvarieties $\K$ of $\mathsf{CM}$ have the weak epimorphism surjectivity property.\ For let $\A \leq \B \in \K$  be finitely generated. 
 As every proper subvariety of $\CMon$ is locally finite (see, e.g., \cite[p.\ 23]{MR2581451}),
the commutative  monoid $\A$ is finite.
Consequently, either $A = B$ or $\mathsf{d}_\K(\A, \B) \ne B$ by a result of Howie and Isbell (see \cite[Thm.\ 3.2]{HoIsEpiII}).\footnote{The statement of \cite[Thm.\ 3.2]{HoIsEpiII} is about commutative semigroups. However, in light of Corollary \ref{Cor : compute doms of subvarieties in CMon} and Proposition \ref{Prop : dom computed in CSG vs CMon}  and the observation that monoids have the same ideals as their semigroup reducts, the corresponding statement for varieties of commutative monoids holds as well.} Hence, $\K$ has the weak epimorphism surjectivity property, as desired. Notice that this property fails for $\mathsf{CM}$ itself, as the commutative monoids $\mathbb{N} = \langle \mathbb{N}, +, 0 \rangle$ and $\mathbb{Z} = \langle \mathbb{Z}, +, 0 \rangle$ are finitely generated and $\mathsf{d}_{\mathsf{CM}}(\mathbb{N}, \mathbb{Z}) = \mathbb{Z}$ (the latter because inverses are preserved by monoid homomorphisms). 

These observations are summarized in the table below, where ``ES'' stands for ``epimorphism surjectivity property''. 
\qed
\end{Remark}

\begin{tabular}{|c | c | c | c |}
\hline
& inverse subvarieties & proper noninverse subvarieties &  \, \, $\mathsf{CM}$ \, \, \\
\hline
weak ES & $\checkmark$ & $\checkmark$  & $\times$\\
\hline
ES & $\checkmark$ & $\times$  & $\times$\\
\hline
strong ES & $\checkmark$ & $\times$  & $\times$\\
\hline
Beth companion & $\checkmark$ & $\times$  & $\times$\\
\hline
\end{tabular}

\

\subsection*{Acknowledgments} The second author was supported by the ayuda PREP$2022$-$000927$ financiada por MICIU/AEI/$10$.$13039$/$501100011033$ y por FSE+ as part of the proyecto PID$2022$-$141529$NB-C$21$ de investigación financiado por MICIU/AEI/$10$.$13039$/$501100011033$ y por FEDER,
UE.
The third author was supported by the proyecto PID$2022$-$141529$NB-C$21$ de investigación
financiado por MICIU/AEI/$10$.$13039$/$501100011033$ y por FEDER, UE. He was also supported
by the MSCA-RISE-Marie Skłodowska-Curie Research and Innovation Staff Exchange (RISE)
project MOSAIC $101007627$ funded by Horizon $2020$ of the European Union.


\begin{thebibliography}{KMPT82}

\bibitem[AHS06]{AHS06}
J.~Ad{\'a}mek, H.~Herrlich, and G.~E. Strecker.
\newblock Abstract and concrete categories: the joy of cats.
\newblock {\em Repr. Theory Appl. Categ.}, (17):1--507, 2006.

\bibitem[Ber11]{Ber11}
C.~Bergman.
\newblock {\em Universal Algebra: Fundamentals and Selected Topics}.
\newblock Chapman \& Hall Pure and Applied Mathematics. Chapman and Hall/CRC, 2011.

\bibitem[Ber15]{Ber15}
G.~M. Bergman.
\newblock {\em An invitation to general algebra and universal constructions}.
\newblock Universitext. Springer, Cham, second edition, 2015.

\bibitem[BMR17]{BMRES}
G.~Bezhanishvili, T.~Moraschini, and J.~G. Raftery.
\newblock Epimorphisms in varieties of residuated structures.
\newblock {\em J. Algebra}, 492:185--211, 2017.

\bibitem[BS12]{BuSa00}
S.~Burris and H.~P. Sankappanavar.
\newblock {\em A Course in Universal Algebra}.
\newblock 2012.
\newblock The millennium edition, available online.

\bibitem[CD90]{CD90}
L.~Czelakowski and W.~Dziobiak.
\newblock Congruence distributive quasivarieties whose finitely subdirectly irreducible members form a universal class.
\newblock {\em Algebra Universalis}, 27(1):128--149, 1990.

\bibitem[CK90]{ModCK}
C.~C. Chang and H.~J. Keisler.
\newblock {\em Model theory}, volume~73 of {\em Studies in Logic and the Foundations of Mathematics}.
\newblock North-Holland Publishing Co., Amsterdam, third edition, 1990.

\bibitem[CKM25a]{CKMWES}
L.~Carai, M.~Kurtzhals, and T.~Moraschini.
\newblock Epimorphisms between finitely generated algebras.
\newblock {\em Indag. Math. (N.S.)}, 36(5):1336--1354, 2025.

\bibitem[CKM25b]{CKMIMP}
L.~Carai, M.~Kurtzhals, and T.~Moraschini.
\newblock The theory of implicit operations.
\newblock Available at \url{https://arxiv.org/abs/2512.14326v1}, 2025.

\bibitem[CKM26]{CKMRings}
L.~Carai, M.~Kurtzhals, and T.~Moraschini.
\newblock Implicit operations in reduced commutative rings.
\newblock In preparation, 2026.

\bibitem[GLV22]{GLVVCMon}
S.~V. Gusev, Edmond W.~H. Lee, and Boris~M. Vernikov.
\newblock The lattice of varieties of monoids.
\newblock {\em Jpn. J. Math.}, 17(2):117--183, 2022.

\bibitem[Gri77]{Gri77}
P.~A. Grillet.
\newblock On subdirectly irreducible commutative semigroups.
\newblock {\em Pacific J. Math.}, 69(1):55--71, 1977.

\bibitem[Gri01]{Gri01}
P.~A. Grillet.
\newblock {\em Commutative semigroups}, volume~2 of {\em Advances in Mathematics (Dordrecht)}.
\newblock Kluwer Academic Publishers, 2001.

\bibitem[HI67]{HoIsEpiII}
J.~M. Howie and J.~R. Isbell.
\newblock Epimorphisms and dominions. {II}.
\newblock {\em J. Algebra}, 6:7--21, 1967.

\bibitem[Hig83]{HigcSG}
P.~M. Higgins.
\newblock The varieties of commutative semigroups for which epis are onto.
\newblock {\em Proc. Roy. Soc. Edinburgh Sect. A}, 94(1-2):1--7, 1983.

\bibitem[HMT85]{HMTII}
L.~Henkin, J.~D. Monk, and A.~Tarski.
\newblock {\em Cylindric algebras. {P}art {II}}, volume 115 of {\em Studies in Logic and the Foundations of Mathematics}.
\newblock North-Holland Publishing Co., Amsterdam, 1985.

\bibitem[How96]{HowZigzag}
J.~M. Howie.
\newblock Isbell's zigzag theorem and its consequences.
\newblock In {\em Semigroup theory and its applications ({N}ew {O}rleans, {LA}, 1994)}, volume 231 of {\em London Math. Soc. Lecture Note Ser.}, pages 81--91. Cambridge Univ. Press, Cambridge, 1996.

\bibitem[Isb66]{Isb65}
J.~R. Isbell.
\newblock Epimorphisms and dominions.
\newblock In {\em Proc. {C}onf. {C}ategorical {A}lgebra ({L}a {J}olla, {C}alif., 1965)}, pages 232--246. Springer-Verlag New York, Inc., New York, 1966.

\bibitem[KMPT82]{SurvKissal}
E.~W. Kiss, L.~M\'{a}rki, P.~Pr\"{o}hle, and W.~Tholen.
\newblock Categorical algebraic properties. {A} compendium on amalgamation, congruence extension, epimorphisms, residual smallness, and injectivity.
\newblock {\em Studia Sci. Math. Hungar.}, 18(1):79--140, 1982.

\bibitem[Mak99]{MakpBeth}
L.~L. Maksimova.
\newblock Projective {B}eth properties in modal and superintuitionistic logics.
\newblock {\em Algebra Log.}, 38(3):316--333, 379, 1999.

\bibitem[Mal58]{Mal58}
A.~I. Mal'cev.
\newblock On homomorphisms onto finite groups.
\newblock {\em Ivanov. Gos. Ped. Inst. Uc. Zap.}, 18:49--60, 1958.
\newblock (Russian).

\bibitem[ML98]{MR1712872}
S.~Mac~Lane.
\newblock {\em Categories for the working mathematician}, volume~5 of {\em Graduate Texts in Mathematics}.
\newblock Springer-Verlag, New York, second edition, 1998.

\bibitem[Mor18]{MoAdj}
T.~Moraschini.
\newblock A logical and algebraic characterization of adjunctions between generalized quasi-varieties.
\newblock {\em J. Symb. Log.}, 83(3):899--919, 2018.

\bibitem[Rap75]{RapregR}
R.~Raphael.
\newblock Some remarks on regular and strongly regular rings.
\newblock {\em Canad. Math. Bull.}, 17(5):709--712, 1974/75.

\bibitem[SVV09]{MR2581451}
L.~N. Shevrin, B.~M. Vernikov, and M.~V. Volkov.
\newblock Lattices of semigroup varieties.
\newblock {\em Izv. Vyssh. Uchebn. Zaved. Mat.}, (3):3--36, 2009.

\end{thebibliography}

\end{document}